\documentclass[12pt, A4]{article}
\usepackage{amsmath,amssymb,amsfonts,mathrsfs,hyperref,microtype, amsthm}
\usepackage{geometry}
\usepackage[framed,numbered]{matlab-prettifier}
\usepackage[pdftex]{graphicx}
\usepackage{eufrak}
\usepackage{graphicx,color}
\usepackage[T1]{fontenc}
\usepackage{rotating}
\usepackage{booktabs}
\usepackage{mathtools}
\usepackage{float}
\usepackage{breqn}
\usepackage{graphicx}
\usepackage{caption}
\usepackage{multirow}
\usepackage{authblk}
\usepackage{comment}
\usepackage{siunitx}
\newtheorem{thm}{Theorem}[section]
\newtheorem{rem}{Remark}[section]
\newtheorem{mydef}{Definition}[section]
\newtheorem{lem}{Lemma}[section]
\newtheorem{prop}{Proposition}[section]

\numberwithin{equation}{section}

\def\HH{ \EuFrak H}
\def \E{{\rm I\kern-0.16em E}}
\def\P{{\rm I\kern-0.16em P}}
\def\F{{\rm I\kern-0.16em F}}
\def\B{{\rm I\kern-0.16em B}}
\def\C{{\rm I\kern-0.46em C}}
\def\G{{\rm I\kern-0.50em G}}
\def\D{{\rm I\kern-0.50em D}}

\newcommand{\R}{\mathbb{R}}
\newcommand{\N}{\mathbb{N}}

\font\eka=cmex10
\usepackage{ae}

\usepackage{color} 
\newcommand{\rg}[1]{{\color{black} #1}}

\def\ind{\mathrel{\hbox{\rlap{%
\hbox to 7.5pt{\hrulefill}}\raise6.6pt\hbox{\eka\char'167}}}}
\begin{document}
\title{ \bf An asymptotic approach to proving sufficiency of Stein characterisations}

\author{Ehsan Azmoodeh 
	\thanks{Department of Mathematical Sciences,
		University of Liverpool, Liverpool L69 7ZL, United Kingdom.
		E-mail: \texttt{ehsan.azmoodeh@liverpool.ac.uk}}, \quad 
	Dario Gasbarra
	\thanks{University of Helsinki, Department of Mathematics and Statistics, B 314, FI-00014, Finland. E-mail: \texttt{dario.gasbarra@helsinki.fi}} \quad Robert E$.$ Gaunt
	 \thanks{The University of Manchester, Department of Mathematics, Alan Turing Building 2.217, M13 9PL Manchester, United Kingdom. E-mail:\texttt{ robert.gaunt@manchester.ac.uk}}
}


\date{ }                     

\setcounter{Maxaffil}{0}
\renewcommand\Affilfont{\itshape\small}

\maketitle 

\vspace{-10mm}

\abstract In extending Stein's method to new target distributions, the first step is to find a Stein operator that suitably characterises the target distribution. In this paper, we introduce a widely applicable technique for proving sufficiency of these Stein characterisations, which can be applied when the Stein operators are linear differential operators with polynomial coefficients. The approach involves performing an asymptotic analysis to prove that only one characteristic function satisfies a certain differential equation associated to the Stein characterisation. We use this approach to prove that all Stein operators with linear coefficients characterise their target distribution, and verify on a case-by-case basis that all polynomial Stein operators in the literature with coefficients of degree at most two are characterising. For $X$ denoting a standard Gaussian random variable and $H_p$ the $p$-th Hermite polynomial, we also prove, amongst other examples, that the Stein operators for $H_p(X)$, $p=3,4,\ldots,8$, with coefficients of minimal possible degree characterise their target distribution, and that the Stein operators for the products of $p=3,4,\ldots,8$ independent standard Gaussian random variables are characterising  (in both settings the Stein operators for the cases $p=1,2$ are already known to be characterising).  We leverage our Stein characterisations of $H_3(X)$ and $H_4(X)$ to derive characterisations of these target distributions in terms of iterated Gamma operators from Malliavin calculus, that are natural in the context of the Malliavin-Stein method.



\section{Introduction}

Stein's method is a powerful technique for bounding the distance between two probability distributions with respect to a given probability metric. It was originally developed for Gaussian approximation by Charles Stein in 1972 \cite{stein}. Since then, Stein's method has been adapted to many other distributions, such as the Poisson \cite{chen0}, beta \cite{dobler beta, goldstein4}, exponential \cite{chatterjee,pekoz1}, gamma \cite{dp18,luk,gpr17}, Laplace \cite{pike} and variance-gamma \cite{gaunt vg}. Stein's method has found applications throughout the mathematical sciences in areas as diverse as random graph theory \cite{bhj92}, queuing theory \cite{bd16} and number theory \cite{h09}. We refer the reader to the monographs \cite{bhj92,chen,n-p-book} and surveys \cite{a21,ross} for detailed accounts of Stein's method and its applications.

In extending Stein's method to a new target distribution $\mu$, the first step is to find a suitable operator $\mathcal{S}$ acting on a class of functions $\mathcal{F}$ that characterises the distribution $\mu$ in the sense that $W\sim \mu$ if and only if $\E[\mathcal{S}f(W)]=0$ for all functions $f\in\mathcal{F}$. This is known as a \emph{Stein characterisation} of $\mu$ and $\mathcal{S}$ is called a \emph{Stein operator} for $\mu$.
For continuous distributions (the subject of this paper), the operator $\mathcal{S}$ is often a linear differential operator with polynomial coefficients, a \emph{polynomial Stein operator}. For the standard Gaussian distribution $N(0,1)$, the classical Stein operator is $\mathcal{S}=\partial-y$, which acts on the class of all absolutely continuous functions $f:\R\rightarrow\R$ such that $\E|f'(X)|<\infty$ for $X\sim N(0,1)$. Here and throughout the paper, $\partial$ denotes the usual differential operator.

There is now a quite extensive literature (for an overview see \cite{a21,ley}) on the problem of finding Stein operators for new target distributions, which establishes necessity of the characterisation and inolves finding an operator $\mathcal{S}$ acting on a class of functions $\mathcal{F}$ such that, if $W\sim \mu$, then $\E[\mathcal{S}f(W)]=0$ for all $f\in\mathcal{F}$. Stein operators have been found for target distributions as complex as the product of $p\geq1$ independent standard Gaussian random variables \cite{gaunt-pn}, linear combinations of $p\geq1$ independent gamma distributions \cite{aaps19b}, and an algorithm has been obtained that yields (amongst other target distributions) all Stein operators for $H_p(X)$, $p\geq1$, where $X\sim N(0,1)$ and $H_p(x)=(-1)^p\mathrm{e}^{x^2/2}\rg{\partial^p}(\mathrm{e}^{-x^2/2})$ is the $p$-th Hermite polynomial \cite{agg19}.

There are also several standard approaches to proving sufficiency of a Stein characterisation, which we briefly review. Let $Y\sim \mu$. All Stein operators obtained by the generator method \cite{barbour1,barbour2,gotze} (recognising $\mathcal{S}$ as the generator of a Markov process with stationary distribution $\mu$) are characterising by basic theory of Markov processes. If $\mu$ is determined by its moments, then plugging $f_k(y)=y^k$, $k=0,1,2,\ldots$, into $\E[\mathcal{S}f(W)]=0$ yields a moment sequence for $W$ which if shown to be equal to the moment sequence of $Y\sim \mu$ proves that $W\sim \mu$ (see the proof of Lemma 5.2 of \cite{ross} for an example). If $\mu$ is not necessarily determined by its moments, one can consider the \emph{Stein equation} $\mathcal{S}f_z(y)=h_z(y)-\E[h(Y)]$ with indicator test functions $h_z(y)=\mathbf{1}(y\leq z)$. If it can be proved that the solution $f_z\in\mathcal{F}$ for all $z\in\R$, then $0=\E[\mathcal{S}f_z(W)]=\P(W\leq z)-\P(Y\leq z)$ for all $z\in\R$, meaning that $\mathcal{L}(W)=\mathcal{L}(Y)$. For an example of this apporach see the proof of Lemma 2.1 of \cite{chen}. This approach has also been used to established Stein characterisations for wide classes of distributions by \cite{ls13,schoutens,stein04}.   Another approach involves taking $f(y)=\mathrm{e}^{\mathrm{i}ty}$ in  $\E[\mathcal{S}f(W)]=0$ and solving the resulting ordinary differential equation (ODE) for the characteristic function of $W$ and deducing that it is equal to the characteristic function of $Y\sim \mu$ (see Exercise 3.8.1 of \cite{n-p-book} for an example). 

The above techniques can be used to arrive at simple proofs of sufficiency of Stein characterisations for a number of classical distributions such as the Gaussian, beta and gamma. However, there exist many distributions for which Stein operators have been obtained, but which have yet to be proved to be characterising and for which the above techniques are not applicable. As an example, which was the original motivation for this paper, consider the Stein operators for the Gaussian Hermite polynomials $H_p(X)$, $p\geq3$ and $X\sim N(0,1)$, some of which are collected in Appendix \ref{appendixa}. These Stein operators were recently obtained by \cite{agg19}, and represent a first step towards extending the Malliavin-Stein method for target distributions of the form $P(X)$, where $P$ is a polynomial of degree greater than two, a class of target distributions identified by \cite{peccati14} to be of particular importance.  For $p=3$ and $p\geq5$, $H_p(X)$ is not determined by its moments. Also, with the exception of the Stein operator (\ref{h4gen}) for $H_4(X)$, all the Stein operators involve third or higher order differential operators, so the Stein operators cannot be recognised as the generator of a Markov process (there does not exist a generator of order greater than 2 for a stochastic process; see Theorem 1.4 and Example 1.1 of Chapter
7 of \cite{durrett}). Moreover, no solution is known to either the corresponding Stein equations or the ODEs satisfied by the characteristic functions derived from the Stein operator. As such there is a need to develop new techniques for proving that Stein operators characterise their target distributions.

In this paper, we introduce a new approach to proving sufficiency of Stein characterisations. The approach, which is described in detail in Section \ref{sec2}, starts similarly to the characteristic function approach described above in that one plugs $f(y)=\mathrm{e}^{\mathrm{i}ty}$ into $\E[\mathcal{S}f(W)]=0$ to find an ODE satisfied by the characteristic function of $W$. For distributions such as $H_p(X)$, $p\geq3$, it may not be possible to solve the ODE exactly. Our approach, which we call the \emph{asymptotic approach} to proving sufficiency of Stein characterisations, is to perform an asymptotic analysis to prove that there can only be one solution that satisfies the properties of characteristic functions; we know that this characteristic function must be that of the target distribution $\mu$, and so we conclude that $W\sim \mu$. 

\rg{In this paper, we restrict our attention to absolutely continuous univariate distributions; however, our basic approach applies to other types of distributions. For absolutely continuous multivariate distributions, on $\mathbb{R}^d$ say, for which Stein operators are partial differential operators, we would take $f(y)=\mathrm{e}^{\mathrm{i}t^\intercal y}$, for $t,y\in\mathbb{R}^d$ to obtain a PDE satisfied by the characteristic function of the target random variable. One would then perform an asymptotic analysis on the solutions of this PDE to show that there can only be one solution that satisfies the properties of characteristic functions. For discrete univariate distributions, for which Stein operators are difference operators, we would again take $f(y)=\mathrm{e}^{\mathrm{i}ty}$ to obtain a difference equation satisfied by the characteristic function of the target random variable. There is also a well-developed theory for performing asymptotic analysis on solutions of difference equations; see \cite[Chapter 5]{bender}.}

In Section \ref{sec3}, we use the asymptotic approach to prove that a number of Stein operators from the literature are characterising. We begin the section, however, by giving a simple example of a Stein operator that is not characterising unless a suitable additional condition is given (Section \ref{secnon}). The characteristic function approach to proving Stein characterisations allows us to efficiently construct such an example, which, as far as we are aware, is the first example in the literature of a Stein operator that is not characterising unless suitable additional conditions are specified. The take home message is that one cannot expect a Stein operator to be characterising. The rest of the section is devoted to establishing Stein characterisations. We show that all polynomial Stein operators with linear coefficients are characterising, and are able to verify on a case-by-case basis that all polynomial Stein operators in the literature with coefficients of degree at most two are characterising. We also prove that the Stein operators of \cite{agg19} for $H_p(X)$, $p=3,4,\ldots,8$, with coefficients of minimal possible degree characterise their target distribution, as well as the Stein operators for $H_3(X)$ and $H_4(X)$ with lowest possible order of derivatives.  The Stein operators for the products of $p=3,4,\ldots,8$ independent standard Gaussian random variables are also shown to be characterising. In each of these settings, it is known in the cases $p=1,2$ that the Stein operators are characterising, and this can in fact easily be proved because if $X_1$ and $X_2$ are independent standard Gaussian random variables, the distributions of $H_1(X_1)$, $H_2(X_1)$, $X_1$ and $X_1X_2$ are determined by their moments \cite{Slud}. The distributions of $H_p(X_1)$, $p=3$, $p\geq5$ and the product of $p\geq3$ standard Gaussian random variables are, however, not determined by their moments \cite{Slud,sld14}; the distribution of $H_4(X)$ is determined by its moments \cite{Slud}. 


Our proofs in Section \ref{sec3} utilise classical techniques for obtaining asymptotic approximations of solutions to ODEs in the neighbourhood of singularities (see Chapter 3 of \cite{bender} for a detailed account of these methods). These proofs may serve as useful illustrative examples for practitioners of Stein's method wishing to use the asymptotic approach to proving sufficiency of Stein characterisations in their own research. To aid such readers we have provided Remark \ref{remcharcc} which discusses some of the difficulties in applying the asymptotic approach and explains how there may be situations in which it will break down.  

In Section \ref{sec4}, we leverage our Stein characterisations of $H_3(X)$ and $H_4(X)$ to establish ``Gamma characterisations" of these distributions (Proposition \ref{prop:Gamma-Type-Expressions}), that is characterisations of the target distributions in terms of iterated gamma operators from Malliavin calculus \cite{np10}. Gamma characterisations are often utilised in proofs of quantitative limit theorems on Wiener space via the Nourdin-Peccati Malliavin-Stein method \cite{np09,n-p-book}. As we elaborate on in Remark \ref{sec4rem}, our Gamma characterisations of $H_3(X)$ and $H_4(X)$ differ in two major features in comparison with Gamma characterisations of random variables belonging to the second Wiener chaos. This provides insight into the Malliavin-Stein method for target distributions belonging to third and higher order Wiener chaoses, and together with the rather complex Stein operators for $H_p(X)$, $p\geq3$, suggest that applying the Malliavin-Stein method to such target distributions is significantly more difficult than for distributions for the first and second Wiener chaoses.  \rg{ Furthermore, we illustrate a potential application of our results in the setup of asymptotic theory of $U$-statistics, see e.g., \cite{Lee_Ustatistics}.}


\vspace{3mm}

\noindent{\emph{Note on the class of functions $\mathcal{F}$}:} Consider the polynomial Stein operator $\mathcal{S}=\sum_{t=0}^Tp_t(y)\partial^t$, with $\max_{0\leq t\leq T}\mathrm{deg}(p_t(y))=m$, for the target random variable $Y$, supported on $I\subseteq\R$. Throughout this paper, unless otherwise stated, the class of functions on which our polynomial Stein operators act is the class $\mathcal{F}_{\mathcal{S},Y}$, which is defined to be the set of all functions $f\in C^T(I)$ such that $\E|Y^jf^{(t)}(Y)|<\infty$ for all $t=0,\ldots,T$ and $j=0,\ldots,m$.  
We do not claim that this is the largest class of functions on which the Stein operators used in this paper act, but the class guarantees that $\E[\mathcal{S}f(Y)]=0$, $\forall f\in\mathcal{F}$, and the class also contains the real and imaginary parts of the functions $f(y)=\mathrm{e}^{\mathrm{i}ty}$, $t\in\R$. 

\section{Description of the approach and first results}\label{sec2}

In this section, we provide an exposition of our asymptotic approach to proving sufficiency of Stein characterisations. We also show how the approach can be used to prove that certain classes of Stein operators characterise their target distributions (Theorems \ref{linearprop} and \ref{prop2ndode}). 
We begin by noting the following simple lemma.

\begin{lem}\label{lemcf}Let $\mathcal{S}=\sum_{i=0}^m\sum_{j=0}^T a_{i,j}y^i\partial^j$ be a Stein operator for the continuous random variable $Y$, with bounded absolute $m$-th moment and support on $I\subseteq\R$, in that 
	\begin{equation}\label{Aexpect}\E  [\mathcal{S}f(Y)]=0
	\end{equation}
	for all $f\in\mathcal{F}\subseteq\mathcal{G}$, where $\mathcal{G}$ is the class of all $f\in C^T(I)$ such that $\E  |\mathcal{S}f(Y)|<\infty$.  Then $\phi_Y(t)=\E  [\mathrm{e}^{\mathrm{i}tY}]$, the characteristic function of $Y$, is a solution to the 
	ODE
\begin{equation}\label{cfode}\sum_{i=0}^m\sum_{j=0}^T a_{i,j}\mathrm{i}^{j-i}t^j\phi^{(i)}(t)=0.
\end{equation}
If the only solution to (\ref{cfode}) that has the property of being a characteristic function of a real-valued random variable is $\phi_Y(t)=\E  [\mathrm{e}^{\mathrm{i}tY}]$, then the following converse holds: Suppose $W$ is a real-valued random variable with $\E  [|W|^m]<\infty$. If $\E  [\mathcal{S}f(W)]=0$ for all $f\in\mathcal{F}$, then $W$ is equal in law to $Y$.
\end{lem}

\begin{proof}That $\phi_Y(t)$ satisfies the ODE (\ref{cfode}) follows from setting $f(y)=\mathrm{e}^{\mathrm{i}ty}$ in (\ref{Aexpect}) and using that $\phi_Y^{(k)}(t)=\mathrm{i}^k\E  [Y^k\mathrm{e}^{\mathrm{i}tY}]$ (with the first $m$ derivatives existing because $\E  [|Y|^m]<\infty$).  It should be noted that $f(y) = \mathrm{e}^{\mathrm{i}ty}$ is complex-valued; here we have applied (\ref{Aexpect}) to the real and imaginary parts of $f$.
Similarly, since  $\E  [|W|^m]<\infty$ and $\E  [\mathcal{S}f(W)]=0$ for all $f\in\mathcal{F}$, it follows that $\phi_W(t)=\E  [\mathrm{e}^{\mathrm{i}tW}]$, the characteristic function of $W$, also satisfies the ODE (\ref{cfode}). If the only solution to (\ref{cfode}) that has the property of being a characteristic function of a real-valued random variable is $\phi_Y(t)$, then it follows that $\phi_W(t)=\phi_Y(t)$, and so by the uniqueness of characteristic functions we conclude that $W$ is equal in law to $Y$.
\end{proof}

An immediate consequence is that all Stein operators with linear coefficients for random variables with bounded absolute first moment are characterising.  The result is a reformulation of part of Lemma 2.1 of \cite{aaps19b}.

\begin{thm}\label{linearprop}Let $\mathcal{S}=\sum_{j=0}^T(a_{0,j}+a_{1,j}y)\partial^j$ be a Stein operator, acting on the class of functions $\mathcal{F}_{\mathcal{S},Y}$, for the continuous random variable $Y$, with bounded absolute first moment. Suppose $W$ is a real-valued random variable with $\E  |W|<\infty$. If $\E  [\mathcal{S}f(W)]=0$ for all $f\in\mathcal{F}_{\mathcal{S},Y}$, then $W$ is equal in law to $Y$. 
\end{thm}

\begin{proof}By Lemma \ref{lemcf} and assumption, $\phi_Y(t)=\E  [\mathrm{e}^{\mathrm{i}tY}]$ and $\phi_W(t)=\E  [\mathrm{e}^{\mathrm{i}tW}]$ are solutions to the boundary value problem
\begin{equation}\label{cfode2}\bigg(\sum_{j=0}^Ta_{1,j}\mathrm{i}^{j-1}t^j\bigg)\phi'(t)+\bigg(\sum_{j=0}^Ta_{0,j}\mathrm{i}^jt^j\bigg)\phi(t)=0, \quad \phi(0)=1.
\end{equation}
By the uniqueness of boundary value problems for first order homogeneous linear ODEs, the unique solution to (\ref{cfode2}) is given by $\phi_Y(t)=\phi_W(t)$, and so $W$ is equal in law to $Y$.  
\end{proof}

\begin{rem}Many Stein operators in the literature have linear coefficients; examples include the Gaussian \cite{stein}, gamma \cite{dz,luk}, variance-gamma \cite{gaunt vg} and McKay Type I distributions \cite{aaps19b}, as well as the product of two independent standard Gaussians \cite{gaunt-pn}  and linear combinations of gamma random variables \cite{aaps19b}.  All these Stein operators are already known to characterise the distribution through several other approaches.  Additionally, the Stein operators of \cite{gms19b} for the product of two independent Gaussian random variables with possibly non-zero means, and the Stein operators of \cite{gaunt-34} for the noncentral chi-square distribution and the distribution of $aX^2+bX+c$, $a,b,c\in\R$ and $X\sim N(0,1)$, have linear coefficients and thus characterise the distribution, a fact that was not noted in these works.
\end{rem}

As (\ref{cfode}) is a $m$-th order homogeneous linear ODE, the general solution can be expressed as
\begin{equation}\label{homsoln}\phi(t)=C_1\phi_1(t)+\cdots+C_{m-1}\phi_{m-1}(t)+C_m\phi_Y(t),
\end{equation}
where $\phi_1,\ldots,\phi_{m-1}$ and $\phi_Y$ are linearly independent solutions to (\ref{cfode}) and $C_1,\ldots,C_m$ are arbitrary constants.  If one can prove that the only way that $\phi(t)$, as given in (\ref{homsoln}), defines a characteristic function is for $C_1=\cdots=C_{m-1}=0$ then uniqueness has been established, since the condition $\phi(0)=1$ forces $C_m=1$. In this paper, we shall establish uniqueness by carrying out an asymptotic analysis in the limit $t\rightarrow0$.     By making use of the facts that the characteristic function $\phi_U(t)$ of a random variable $U$ is bounded at $t=0$, that if $\E[|U|^k]<\infty$ then $\phi_U^{(k)}(0)=\E[U^k]$, and that if $U$ is symmetric ($U=_d-U$) then $\phi_U(t)$ is a real-valued function of $t$, we are then able to prove that we must take $C_1=\cdots=C_{m-1}=0$. This approach is particularly powerful for Stein operators with polynomial coefficients of degree at most 2 (see Section \ref{sec3.1}), in which case we only need to prove that $C_1=0$. 

We remark that to prove that certain Stein operators are characterising it may be the case that an analysis as $t\rightarrow0$ does not allow one to prove sufficiency. In such cases, one may need to perform the analysis at another point, such as $t=\infty$, in which case identifying unbounded solutions, for example, may lead to a proof of sufficiency. A standard method of examining the behaviour of solutions to ODEs in the neighbour of a singularity at infinity involves making a change of variables $t=1/w$ and then performing an asymptotic analysis in the limit $w\rightarrow0$. This change of variables leads to longer calculations. Moreover, of the Stein operators we consider in Section \ref{sec3} there are examples in which we were able to deduce sufficiency from an analysis in the limit $t\rightarrow0$ but not as $t\rightarrow\infty$; of course there may be Stein operators for which the opposite is true.




In the following theorem, we establish some conditions under which a Stein operator with quadratic coefficents characterises the distribution, which will be used to prove items (i) and (iii) of Propostion \ref{propngb} and items (iii) and (iv) of Proposition \ref{prophpchar}.  The proof of the theorem involves an application of classical techniques for obtaining asymptotic approximations for solutions of ODEs in the neighbourhood of essential singularities; a detailed account of the techniques is given in Chapter 3 of \cite{bender}. 

  
\begin{thm}\label{prop2ndode}Let $\mathcal{S}=\sum_{j=0}^T (a_{0,j}+a_{1,j}y+a_{2,j}y^2)\partial^j$ be a Stein operator, acting on the class of functions $\mathcal{F}_{\mathcal{S},Y}$, for the continuous random variable $Y$ with $\E  [Y^2]<\infty$.  Consider the associated ODE
\begin{equation}\label{cfode2nd}\phi''(t)+p(t)\phi'(t)+q(t)\phi(t)=0,
\end{equation}
where
\begin{align*}p(t)=\frac{\sum_{j=0}^T a_{1,j}\mathrm{i}^{j-1}t^j}{\sum_{j=0}^T a_{2,j}\mathrm{i}^{j-2}t^j}, \quad q(t)=\frac{\sum_{j=0}^T a_{0,j}\mathrm{i}^{j}t^j}{\sum_{j=0}^T a_{2,j}\mathrm{i}^{j-2}t^j}.
\end{align*} Suppose that either
\begin{itemize}
\item [(i)] The function $p(t)$ has a pole at $t=0$ of order $\alpha$, where $\alpha\geq3$ is an odd number, with $\lim_{t\rightarrow0}t^\alpha p(t)=p_0>0$.

\item [(ii)]  The function $p(t)$ has the following asymptotic behaviour:
\[p(t)\sim\frac{a\mathrm{i}}{t^2}+\frac{b}{t}, \quad t\rightarrow0,\]
where $a\in\R\setminus\{0\}$ and $b\geq-2$.
\end{itemize}
\rg{Then the following holds:}
Suppose $W$ is a real-valued random variable with $\E  [W^2]<\infty$.  If $\E  [\mathcal{S}f(W)]=0$ for all $f\in\mathcal{F}_{\mathcal{S},Y}$, then $W$ is equal in law to $Y$.
\end{thm}

\begin{rem} Theorem \ref{prop2ndode} covers only a fraction of the classes of Stein operators with quadratic coefficients to which the general approach described in this section can be applied.  Other classes include Stein operators whose associated characteristic function ODEs have essential singularities with even order greater than 2, regular singular points, singularities at points other than the origin, and ODEs in which $p(t)$ and $q(t)$ are both analytic.  As different asymptotic formulas are available in each of these cases, for reasons of brevity, we do not explore this further in this paper.
For polynomial Stein operators with quadratic coefficients that do not fall into the setting of Theorem \ref{prop2ndode}, we establish sufficiency on a case-by-case basis in Section \ref{sec3.1}.
\end{rem}

\begin{proof}The basic approach we shall take, in cases (i) and (ii), is to prove, though an asymptotic approach, that there is a solution to (\ref{cfode2nd}) that is either unbounded at $t=0$ or not twice-differentiable at $t=0$ (note that we assume that $\E  [W^2]<\infty$ and so the characteristic function of $W$ must be twice-differentiable).  Given this behaviour, these solutions must be linearly independent of another solution $\phi_Y(t)$, the characteristic function of $Y$, and we can thus conclude that the only solution to (\ref{cfode2nd}) satisfying the properties of characteristic functions is $\phi_Y(t)$, and applying Lemma \ref{lemcf} then proves the lemma.

 In both cases (i) and (ii), the pole at $t=0$ is of order at least 2 and so there is an essential singularity at the origin.  Thus, we seek an approximation for one of the solutions using the ansatz $\phi(t)=\mathrm{e}^{S(t)}$ (see \cite[Section 3.4]{bender}). Substituting into (\ref{cfode2nd}) gives that
\begin{equation*}S''(t)+(S'(t))^2+p(t)S'(t)+q(t)=0.
\end{equation*}
 We now make the standard assumption that $S''(t)\ll(S'(t))^2$, $t\rightarrow 0$, (again, see \cite[Section 3.4]{bender}); which we shall verify later in the proof, from which we obtain
\begin{equation*}(S'(t))^2+p(t)S'(t)+q(t)\sim0, \quad t\rightarrow0.
\end{equation*}
The quadratic formula then gives
\[S'(t)\sim\frac{1}{2}\Big(-p(t)\pm\sqrt{p(t)^2-4q(t)}\Big), \quad t\rightarrow0.
\]
We now make the observation that the order of the pole of $p(t)$ at $t=0$ must be at least as great as the order of the pole of $q(t)$ at $t=0$.  This is because $\phi_Y(0)=1$ and as $\E  [Y^2]<\infty$, $\phi_Y'(t)$ and $\phi_Y''(t)$ are finite in the limit $t\rightarrow0$.  Therefore $p(t)^2\gg q(t)$, as $t\rightarrow0$, and so we have one solution for which $S'(t)=O(1)$ as $t\rightarrow0$ (which is consistent with the behaviour of a characteristic function), and another for which 
\[S'(t)\sim -p(t), \quad t\rightarrow0.\]
Note that in both cases the assumption $S''(t)\ll(S'(t))^2$, $t\rightarrow 0$, is satisfied.

In case (i), we have $S'(t)\sim -p_0/t^\alpha$, $t\rightarrow0$.  On integrating and inserting the result into the ansatz $\phi(t)=\mathrm{e}^{S(t)}$ we obtain the following asymptotic behaviour for a solution to (\ref{cfode2nd}):
\[\phi_1(t)=\exp\bigg(\frac{p_0}{(\alpha-1)t^{\alpha-1}}\big(1+o(1)\big)\bigg), \quad t\rightarrow0.\]
By our assumption $p_0>0$ (note that the assumption that $\alpha$ is an odd number ensures that $p_0$ is a real number), this solution is unbounded in the limit $t\rightarrow0$, and we can therefore conclude that the Stein operator is characterising.

In case (ii), we have $S'(t)\sim -a\mathrm{i}/t^2$, $t\rightarrow0$, for some $a\in\R\setminus\{0\}$.  (The fact that the pole is of even order ensures that the constant $a$ is real-valued.) Proceeding as in part (i), we obtain the following asymptotic behaviour for a solution to (\ref{cfode2nd}):
\[\phi_1(t)=\exp\bigg(\frac{a\mathrm{i}}{t}\big(1+o(1)\big)\bigg), \quad t\rightarrow0.\] 
In this case, we cannot yet conclude that $\phi_1$ is either not twice-differentiable or unbounded, and so we must work a little harder than for case (i).  Let us now seek the refined  approximation
\[S(t)=\frac{a\mathrm{i}}{t}+C(t),\]
for some $C(t)\ll 1/t$, $t\rightarrow0$.  Substituting into (\ref{cfode2nd}) gives that
\begin{align*}&\frac{2a\mathrm{i}}{t^3}+C''(t)+\bigg(-\frac{a\mathrm{i}}{t^2}+C'(t)\bigg)^2+p(t)\bigg(-\frac{a\mathrm{i}}{t^2}+C'(t)\bigg)+q(t)=0.
\end{align*}
We can obtain an asymptotic differential equation for $C(t)$ by arguing as follows.  Since $S'(t)\sim -a\mathrm{i}/t^2$, $t\rightarrow0$, we also have that $C'(t)\ll 1/t^{2}$, $t\rightarrow0$, meaning that also $(C'(t))^2\ll C'(t)/t^{2}$, $t\rightarrow0$.  Also, $p(t)\sim a\mathrm{i}/t^2+b/t+c\mathrm{i}$, $t\rightarrow0$, for some $a\in\R\setminus\{0\}$ and $b,c\in\R$, and $q(t)\sim d/t^2$, $t\rightarrow0$, for some $d\in\R$, since $q(t)$ has a pole of order at most 2 at $t=0$.  Such considerations lead to the asymptotic ODE
\begin{align}\label{asyode}\frac{a(2-b)\mathrm{i}}{t^3}+C''(t)-\frac{a\mathrm{i}}{t^2}C'(t)+\frac{ac+d}{t^2}\sim 0, \quad t\rightarrow0.
\end{align}
In the case $b=2$, the only solution to (\ref{asyode}) that satisfies $C(t)\ll 1/t$, $t\rightarrow0$, has the limiting form $C(t)\sim-\frac{ac+d}{a}\mathrm{i}t$, $t\rightarrow0$.  In the case $b\not=2$, it can be readily seen that $C(t)\sim (2-b)\log(t)$, $t\rightarrow0$, satisfies (\ref{asyode}), since $C''(t)\ll 1/t^3$, $t\rightarrow0$, and that this is the only solution that satisfies $C(t)\ll 1/t$, $t\rightarrow0$.  In summary, we have that
\[S(t)\sim\frac{a\mathrm{i}}{t}+(2-b)\log(t), \quad t\rightarrow0,\]
Therefore, we see that there is a solution to (\ref{cfode2nd}) with the following asymptotic behaviour:
\begin{equation*}\phi_1(t)\sim t^{2-b}\mathrm{e}^{a\mathrm{i}/t}, \quad t\rightarrow0.
\end{equation*}
We may differentiate this limiting form to obtain
\[\phi_1''(t)\sim -a^2t^{-b-2}\mathrm{e}^{a\mathrm{i}/t},\quad t\rightarrow0.\]
If $b\geq-2$, then this solution is not twice-differentiable at $t=0$.   The proof is complete.
\end{proof}



\section{Examples}\label{sec3}

\subsection{A non-characterising Stein operator}\label{secnon}
The purpose of this section is to provide a simple example of a Stein operator that is not characterising, unless an additional suitable condition is present.  Consider $X \sim N(0,1)$ a standard Gaussian distribution having characteristic function $\phi_X (t) =\mathrm{e}^{-t^2/2}$ and a semicircular random variable $Y$ with probability density function $f_Y(x) = (2/\pi) \sqrt{1-x^2} \, \textbf{1}_{  [-1,1] }$ supported on the interval $[-1,1]$ with characteristic function $\phi_Y (t) = 2 J_1(t)/t$, where here $J_\nu$ stands for the Bessel function of the first kind of order $\nu$. One should note that both distributions are moment determined. 

\rg{Let 
\begin{align*}
\mathcal{S}=(1-x^2)\partial^5+(x^3-4x)\partial^4+(5-2x^2)\partial^3+(3x^3-21x)\partial^2+9x2\partial-9x.
\end{align*}
Then, one can readily check using an integration by parts argument that the operator $\mathcal{S}$ is a polynomial Stein operator for both the standard Gaussian and semicircular distributions, namely that 
\[  \E \left[  \mathcal{S} f(X) \right] = \E \left[  \mathcal{S} f(Y) \right] =0, \quad \forall \,  f \in \mathcal{F}, \]
where the class of functions $\mathcal{F}$ contains all functions $f\in C^3(\R)$ such that $\E|X^jf^{(t)}(X)|<\infty$ and $\E|Y^jf^{(t)}(Y)|<\infty$ for all $t=0,\ldots,3$ and $j=0,\ldots,5$, as well as the real and imaginary parts of $f(y)=\mathrm{e}^{\mathrm{i}ty}$, $t\in\R$. Clearly, $\mathcal{S}$ cannot be a characterising Stein operator unless one proposes suitable additional conditions.

That $\mathcal{S}$ is a Stein operator for both the standard Gaussian and semicircular distributions can perhaps be most naturally seen by considering the associated ODE for the characteristic function. By Lemma \ref{lemcf}, the associated characteristic function ODE is
\begin{equation}\label{aode}\mathcal{A}\phi(t)=0,
\end{equation}
where
\[\mathcal{A}= (t^4-3t^2)\partial^3+(t^5-2t^3-9t)\partial^2+(4t^4-21t^2+9)\partial+t^5-5t^3.\]
One can readily check that $\mathcal{A} \phi_X (t) = \mathcal{A} \phi_Y (t) =0$. Through an analysis of the ODE (\ref{aode}) we can provide some additional conditions under which the operator $\mathcal{S}$ is a characterising Stein operator for the either the standard normal or semicircular distributions. The general solution of (\ref{aode}) is given by
\[\phi(t)=C_1\phi_X(t)+C_2\phi_Y(t)+C_3Y_1(t)/t,\]
where $C_1,C_2,C_3\in\mathbb{C}$ are arbitrary constants and $Y_\nu$ stands for the Bessel function of the second kind of order $\nu$. Now, $Y_1(t)/t\sim-(2/\pi)t^{-2}$, as $t\downarrow0$, (see \cite[Section 10.7]{Olver}) so we must take $C_3=0$. If we add the condition that $\E[X^2]=1$, then it is readily seen that we must take $C_2=0$ and $C_1=1$, and we deduce that $\mathcal{S}$ is a characterising Stein operator for the standard normal distribution. On the other hand, impose that $\E[Y^2]=1/4$, then it is readily seen that we must take $C_1=0$ and $C_2=1$, and we deduce that $\mathcal{S}$ is a characterising Stein operator for the semicircular distribution. 

 We stress that other conditions instead of second moment conditions, either involving higher moments or perhaps conditions of a different type like whether the support is bounded, could be included to ensure that $\mathcal{S}$ is a characterising Stein operator, although the key point is that some additional condition must be specified in order to ensure that the Stein operator is characterising, meaning that the characterising property of Stein operators cannot be taken for granted. We also remark that the strategy used to construct this non-characterising Stein operator applies much more widely than to the Gaussian and semicircular distribution; this is explored in detail in our forthcoming work \cite{agg22}.}

\subsection{Polynomial Stein operators with coefficients of maximal degree two}\label{sec3.1}

In this section, we verify on a case-by-case basis that all polynomial Stein operators in the literature with coefficients of maximal degree two are characterising. We begin by noting that a number of such Stein operators are already known to be characterising through other means, these include the generalized inverse Gaussian \cite{kl14}, beta \cite{dobler beta,goldstein4,schoutens}, Student's $t$-distribution \cite{schoutens}, $F$-distribution \cite{gms19}, inverse gamma \cite{gms19} and chi distribution \cite{gaunt-laplace}. The beta, Student's $t$, $F$, inverse gamma and chi distributions are members of the Pearson family, so sufficiency follows from Theorem 1 of \cite{schoutens}. The polynomial Stein operators (\ref{h3x}) and (\ref{h4x}) for $H_3(X)$ and $H_4(X)$, respectively, and the Stein operator (\ref{pnsteinop}) for the product of three independent standard Gaussian random variables also have coefficients with maximal degree two. We prove that these Stein operators are characterising in Propositions \ref{prophpchar} and \ref{propnnn}. In this section, we prove that the other polynomial Stein operators in the literature with coefficients of maximal degree two are characterising.
 
\vspace{3mm}

\noindent{(1)} Consider the PRR distribution (also known as the Kummer distribution) \cite{pekoz} with parameter $s>1/2$ (denoted by $K_s$), and density
\begin{equation*}p(x)=\Gamma(s)\sqrt{\frac{2}{\pi s}}\exp\bigg(-\frac{x^2}{2s}\bigg)U\bigg(s-1,\frac{1}{2},\frac{x^2}{2s}\bigg), \quad x>0,
\end{equation*}
where $U(a,b,x)$ denotes the confluent hypergeometric function of the second kind (see \cite[Chapter 13]{Olver}). A Stein operator (see \cite{pekoz}) is given by 
\begin{align*}\mathcal{S}_sf(y)=sf''(y)-yf'(y)-2(s-1)f(y),
\end{align*}
where the operator acts on twice differentiable functions $f:(0,\infty)\rightarrow\R$ with $f(0)=f'(0)=0$. Setting $g(y)=yf(y)$ (so that the values $g(0)$ and $g'(0)$ do not need to be specified) yields the Stein operator
\begin{align*}\mathcal{S}_s'g(y)=syg''(y) +(2s-y^2)g'(y)+(1-2s)y g(y).
\end{align*}
We now prove the following Stein characterisation, for which necessity is immediate from Lemma 3.1 of \cite{pekoz}, and our contribution is to prove sufficiency.

\begin{prop}\label{propprr} Let $W$ be a real-valued random variable with $\E[W^2]<\infty$. Then $W\sim K_s$ if and only if $\E[\mathcal{S}_s' g(W)]=0$ for all $g\in\mathcal{F}_{\mathcal{S}_s',Y}$, where $Y\sim K_s$.
\end{prop}

\begin{proof}By Lemma \ref{lemcf}, the associated characteristic function ODE is given by 
\begin{equation}\label{lck}t\phi''(t)+(st^2+2s-1)\phi'(t)+2st\phi(t)=0.
\end{equation}
This ODE has a regular singularity at $t=0$, so, in a neighbourhood of $t=0$, we seek a Frobenius series solution $\phi(t)=t^\alpha\sum_{k=0}^\infty c_k t^k$, $ c_0\not=0$. Plugging into (\ref{lck}) shows that $\alpha$ satisfies the indicial equation $\alpha(\alpha-1)+\alpha(2s-1)=0$, that is $\alpha=0$ or $\alpha=2-2s$. The solution $\phi_1(t)$ with $\alpha=0$ is consistent with a characteristic function. Let us now consider the solution $\phi_2(t)$ with $\alpha=2-2s$. If $s>1$, then $\alpha=2-2s<0$ and the solution is unbounded. If $s=1$, then $\alpha=2-2s=0$, so by Fuch's Theorem, $\phi_2(t)=\log(t)\phi_1(t)+\sum_{k=0}^\infty b_kt^{k+r}$, where $b_0\not=0$ and $r\geq0$. Therefore $\phi_2(t)=O(\log(t))$ as $t\rightarrow0$, and is thus unbounded. Finally, if $1/2<s<1$, then $\phi_2'(t)=O(t^{1-2s})$ as $t\rightarrow0$, so the derivative is not well-defined at $t=0$, contradicting the assumption that $\E|W|<\infty$, which is implied by the assumption $\E[W^2]<\infty$. Appealing to Lemma \ref{lemcf} now gives us sufficiency.
\end{proof}

\noindent{(2)} We now consider the problem of proving sufficiency of Stein characterisations for products independent beta, gamma and centered Gaussian random variables. To fix notation, we  denote by $\mathrm{Beta}(a,b)$ the beta distribution with parameters $a,b>0$ and density $p(x)=x^{a-1}(1-x)^{b-1}/B(a,b)$, $0<x<1$, whilst we denote by $\Gamma(r,\lambda)$ the gamma distribution with parameters $r,\lambda>0$ and density $p(x)=\lambda^rx^{r-1}\mathrm{e}^{-\lambda x}/\Gamma(r)$, $x>0$. Let $B\sim\mathrm{beta}(a,b)$, $G_1\sim\Gamma(r,\lambda)$, $G_2\sim\Gamma(s,\lambda)$ and $X\sim N(0,\sigma^2)$ be mutually independent. Then the following Stein operators were obtained by \cite{gaunt-ngb} for the products $G_1X$, $BG_1$ and $G_1G_2$:
\begin{align*}\mathcal{S}_{G_1X}f(y)&=y^2f^{(3)}(y)+2(r+1)yf''(y)+r(r+1)f'(y)-(\lambda/\sigma^2)yf(y), \\
\mathcal{S}_{BG_1}f(y)&=y^2f''(y)+((a+r-1)y-y^2)f'(y)+(ar-(a+b)y)f(y), \\
\mathcal{S}_{G_1G_2}f(y)&=y^2f''(y)+(1+r+s)yf'(y)+(rs-\lambda^2y)f(y).
\end{align*}
Polynomial Stein operators with coefficients of degree at most two are available for the products of two and three independent standard Gaussian random variables. The product of two independent standard Gaussian is a special case of the variance-gamma distribution and sufficiency was established by \cite{gaunt vg}, whilst we consider the characterising problem for the case of three standard Gaussians in Section \ref{sec3.3}. We also note that other polynomial Stein operators are given in \cite{gaunt-ngb} for products of independent beta, gamma and Gaussian random variables, although the maximal degree of their polynomial coefficients is strictly greater than two. The asymptotic approach should be applicable to at least some of these operators, although we do not investigate this further in this paper.

\begin{prop}\label{propngb}
Let $W$ be a real-valued random variable with $\E[W^2]<\infty$. Then

\begin{itemize}

\item [(i)]  $W=_d G_1X$ if and only if $\E[\mathcal{S}_{G_1X}f(W)]=0$ for all $f\in\mathcal{F}_{\mathcal{S}_{G_1X},G_1X}$.

\item [(ii)]   $W=_d BG_1$ if and only if $\E[\mathcal{S}_{BG_1}f(W)]=0$ for all $f\in\mathcal{F}_{\mathcal{S}_{BG_1},BG_1}$.

\item [(iii)]  $W=_d G_1G_2$ if and only if $\E[\mathcal{S}_{G_1G_2}f(W)]=0$ for all $f\in\mathcal{F}_{\mathcal{S}_{G_1G_2},G_1G_2}$.

\end{itemize}
\end{prop}

\begin{proof} 
(i) The associated characteristic function ODE is
\begin{equation*}t^3\phi''(t)+(2(r+1)t^2+\lambda/\sigma^2)\phi'(t)+r(r+1)t\phi(t)=0,
\end{equation*}
and sufficiency immediately follows from part (i) of Theorem \ref{prop2ndode}.

\vspace{2mm} 

\noindent{(ii)} The associated characteristic function ODE is
\begin{equation*}(t^2+\mathrm{i}t)\phi''(t)+((a+r-1)t+\mathrm{i}(a+b))\phi'(t)+ar\phi(t)=0.
\end{equation*}
This ODE has a regular singularity at $t=0$, so we proceed similarly to we did in the proof of Proposition \ref{propprr} for the PRR distribution. Seeking a solution of the form $\phi(t)=t^\alpha\sum_{k=0}^\infty c_kt^k$, $c_0\not=0$, in a neighbourhood of $t=0$ gives the indicial equation $\alpha(\alpha-1)+\alpha(a+b)=0$, so that $\alpha=0$ or $\alpha=1-a-b$. Sufficiency now follows from analysing the cases $a+b>1$, $a+b=1$, $0<a+b<1$ separately, like was done for the different cases of $s$ in the proof of Proposition \ref{propprr};  we omit the details.
\vspace{2mm} 

\noindent{(iii)} This time the associated characteristic function ODE is
\begin{equation*}t^2\phi''(t)+((1+r+s)t+\lambda^2\mathrm{i})\phi'(t)+rs\phi(t)=0,
\end{equation*}
and sufficiency immediately follows from part (ii) of Theorem \ref{prop2ndode}.
\end{proof}

\subsection{Stein operators for Gaussian Hermite polynomials}\label{sec3.2}

In our recent paper \cite{agg19}, we introduced an algorithm that can find all polynomial Stein operators up to a specified order $T$ and maximal polynomial degree $m$ of coefficients for a given Gaussian Hermite polynomial $H_p(X)$, $p\geq1$. In the following proposition, we prove that  the Stein operators for $H_p(X)$, $p=3,\ldots,8$, with minimum possible maximal degree $m$ (see Proposition 4.1 of \cite{agg19}) fully characterise their target distribution. Item (i) of the proposition also asserts that the Stein operator for $H_3(X)$ with $(T,m)=(4,3)$ (which in all likelihood has minimum order $T$ amongst Stein operators with zeroth-order term $cyf(y)$) is characterising.  (The Stein operator (\ref{h4x}) for $H_4(X)$ has both the minimum possible $T$ and $m$ amongst Stein operators with zeroth-order term $cyf(y)$.)  In addition, item (ii) of the proposition asserts that the Stein operator for $H_4(X)$ with $(T,m)=(2,3)$ (which in all likelihood has minimum order $T$ amongst polynomial Stein operators for $H_4(X)$) is characterising. 



\begin{prop}\label{prophpchar}Let $W$ be a real-valued random variable. Let $\mathcal{S}_3'$ and $\mathcal{S}_4'$ denote the Stein operators (\ref{h3x-new}) and (\ref{h4gen}) for $H_3(X)$ and $H_4(X)$ with $(T,m)=(4,3)$ and $(T,m)=(2,3)$, respectively. Also, for $p=3,\ldots,8$, let $\mathcal{S}_p$ be Stein operator for $H_p(X)$ that attains the minimum polynomial coefficient degree $m$.  For $p=3,4,5,6$, these are the Stein operators (\ref{h3x}), (\ref{h4x}), (\ref{h5x}) and (\ref{h6x}), respectively.  The complicated Stein operators for $H_7(X)$ with $(T,m)=(25,8)$ and $H_8(X)$ with $(T,m)=(10,4)$ are given in Appendix B of the arXiv version no.\ 1 of \cite{agg19}.  Then
\begin{itemize}

\item [(i)] Suppose $\E  [|W|^3]<\infty$. Then $W=_d H_3(X)$ if and only if $\E  [\mathcal{S}_3'f(W)]=0$ for all $f\in\mathcal{F}_{\mathcal{S}_3',H_3(X)}$.

\item [(ii)] Suppose $\E[W]=0$ and $\E[|W|^3]<\infty$. Then $W=_d H_4(X)$ if and only if $\E  [\mathcal{S}_4'f(W)]=0$ for all $f\in\mathcal{F}_{\mathcal{S}_4',H_4(X)}$.

\item [(iii)] 
Suppose $\E  [W^2]<\infty$. Then $W=_d H_3(X)$ if and only if $\E  [\mathcal{S}_3f(W)]=0$ for all $f\in\mathcal{F}_{\mathcal{S}_3,H_3(X)}$.

\item [(iv)] 
Suppose $\E  [W^2]<\infty$. Then $W=_d H_4(X)$ if and only if $\E  [\mathcal{S}_4f(W)]=0$ for all $f\in\mathcal{F}_{\mathcal{S}_4,H_4(X)}$.

\item [(v)] 
Suppose $W$ is a symmetric random variable ($W=_d -W$) such that $\E  [W^4]<\infty$. Then $W=_d H_5(X)$ if and only if $\E  [\mathcal{S}_5f(W)]=0$ for all $f\in\mathcal{F}_{\mathcal{S}_5,H_5(X)}$.

\item [(vi)] 
Suppose $\E  [|W|^3]<\infty$. Then $W=_d H_6(X)$ if and only if $\E  [\mathcal{S}_6f(W)]=0$ for all $f\in\mathcal{F}_{\mathcal{S}_6,H_6(X)}$.

\item [(vii)] 
Suppose $W$ is a symmetric random variable such that $\E  [W^6]<\infty$. Then $W=_d H_7(X)$ if and only if $\E  [\mathcal{S}_7f(W)]=0$ for all $f\in\mathcal{F}_{\mathcal{S}_7,H_7(X)}$.

\item [(viii)] 
Suppose $\E  [W^4]<\infty$. Then $W=_d H_8(X)$ if and only if $\E  [\mathcal{S}_8f(W)]=0$ for all $f\in\mathcal{F}_{\mathcal{S}_8,H_8(X)}$.

\end{itemize}
\end{prop}

\begin{proof} Necessity of all the characterisations was established by \cite{agg19}.  For the Stein operators $\mathcal{S}_3$ and $\mathcal{S}_4$ for $H_3(X)$ and $H_4(X)$, necessity was also earlier established in Propositions 2.4 and 2.5 of \cite{gaunt-34}.  Let us now establish sufficiency. For the Stein operators $\mathcal{S}_3$ and $\mathcal{S}_4$ we appeal to Theorem \ref{prop2ndode}.  According to Lemma \ref{lemcf}, the associated characteristic function ODEs for the Stein operators $\mathcal{S}_k$, $k=3,4$, are given by
\[\phi''(t)+p_k(t)\phi'(t)+q_k(t)\phi(t)=0,\]
where
\begin{align*}p_3(t)&=\frac{1+99t^2+486t^4}{27t^3+486t^5}\sim\frac{1}{27t^3}, \quad t\rightarrow0, \\
p_4(t)&=\frac{\mathrm{i}+44t-144\mathrm{i}t^2+576t^3}{16t^2-192\mathrm{i}t^3}\sim\frac{\mathrm{i}}{16t^2}+\frac{2}{t}, \quad t\rightarrow0.
\end{align*}
In each case either condition (i) or (ii) of Theorem \ref{prop2ndode} is satisfied and so we conclude that the Stein operators are characterising.

Proving sufficiency for the other Stein operators is more involved, and is worked out on a case-by-case basis.  

\vspace{2mm}  

\noindent{(i)} By Lemma \ref{lemcf}, the associated characteristic function ODE for $\mathcal{S}_3'$ is given by
\begin{align}\label{h3dash}81t^4\phi^{(3)}(t)+(351t^3+3t)\phi''(t)+(324t^4+207t^2-5)\phi'(t)+(1080t^3-12t)\phi(t)=0.
\end{align}
This ODE has an essential singularity at $t=0$, so we use the ansatz $\phi(t)=\mathrm{e}^{S(t)}$.  By the method of dominant balance we obtain the following asymptotic ODE for $S(t)$:
\begin{equation}\label{sh3eqn}81(S'(t))^3+\frac{3}{t^3}(S'(t))^2-\frac{5}{t^4}S(t)\sim0, \quad t\rightarrow0.
\end{equation}
It is straightforward to verify that a pair of linearly independent solutions to the ODE $81t^4(f'(t))^2+3tf'(t)-5=0$ are given by
\begin{eqnarray*}f_1(t)&=&\frac{1+\sqrt{1+180t^2}}{108t^2}-\frac{5}{3}\log\bigg(\frac{t}{1+\sqrt{1+180t^2}}\bigg)\sim \frac{1}{54t^2}, \quad t\rightarrow0, \\
f_2(t)&=&\frac{1-\sqrt{1+180t^2}}{108t^2}+\frac{5}{3}\log\bigg(\frac{t}{1+\sqrt{1+180t^2}}\bigg)\sim \frac{5}{3}\log(t), \quad t\rightarrow0.
\end{eqnarray*}
The asymptotic ODE (\ref{sh3eqn}) therefore has three linearly independent solutions with the following leading order terms in the limit $t\rightarrow0$,
\begin{align*}S_1(t)\sim c_1, \quad S_2(t)\sim \frac{1}{54t^2}, \quad S_3(t)\sim \frac{5}{3}\log(t),
\end{align*}
where $c_1$ is a constant. We therefore have that there are three linearly independent solutions to (\ref{h3dash}) with the following asymptotic behaviour as $t\rightarrow0$:
\begin{align*}\phi_1(t)=O(1), \quad \phi_2(t)= \exp\bigg(\frac{1}{54t^2}(1+o(1))\bigg), \quad \phi_3(t)\sim t^{5/3}.
\end{align*}
The solution $\phi_2(t)$ blows up as $t\rightarrow0$, and the second derivative of $\phi_3(t)$ blows up as $t\rightarrow0$ (in contradiction to the finite second moment assumption).  We complete the proof of sufficiency by appealing to Lemma \ref{lemcf}.

\vspace{2mm}  

\noindent{(ii)}  The associated characteristic function ODE for $\mathcal{S}_4'$ is given by
\begin{align}&16t^2\phi^{(3)}(t)+(-48\mathrm{i}t^2+64t+\mathrm{i})\phi''(t)\nonumber\\
\label{h4dash}&+(576t^2+72\mathrm{i}t+50)\phi'(t)+(-1728\mathrm{i}t^2+1008t+24\mathrm{i})\phi(t)=0.
\end{align}
This ODE has an essential singularity at $t=0$, so we use the ansatz $\phi(t)=\mathrm{e}^{S(t)}$, and using the method of dominant balance gives the following asymptotic ODE for $S(t)$:
\begin{align*}(S'(t))^3+\frac{\mathrm{i}}{16t^2}(S'(t))^2\sim0, \quad t\rightarrow0.
\end{align*}
The are therefore three linearly independent solutions to (\ref{h4dash}) with the following asymptotic behaviour as $t\rightarrow0$,
\begin{equation*}\phi_1(t)=O(1), \quad \phi_2(t)=O(1), \quad \phi_3(t)=\exp\bigg(\frac{\mathrm{i}}{16t}(1+o(1))\bigg).
\end{equation*}

We now further analyse the asymptotic behaviour of the solution $\phi_3(t)$ as $t\rightarrow0$. We proceed as in the proof of part (ii) of Theorem \ref{prop2ndode} and seek the refined approximation
\begin{align*}S_3(t)=\frac{\mathrm{i}}{16t}+C(t),
\end{align*}
where $C(t)\ll 1/t$, as $t\rightarrow0$.  The method of analysis used to obtain asymptotic approximations for the correction terms $C_2(t)$ and $C_4(t)$ is similar to the one in the proof of part (ii) of Theorem \ref{prop2ndode}, and as such our analysis will not be as detailed. We have that $S_3'(t)=-(\mathrm{i}/16)t^{-2}+C'(t)$, $S_3''(t)=(\mathrm{i}/8)t^{-3}+C''(t)$ and $S_3'(t)=-(3\mathrm{i}/8)t^{-4}+C^{(3)}(t)$, and therefore, as $t\rightarrow0$,
\begin{align*}&16\bigg[\frac{-3\mathrm{i}}{8t^4}+C^{(3)}(t)+3\bigg(-\frac{\mathrm{i}}{16t^2}+C'(t)\bigg)\bigg(\frac{\mathrm{i}}{8t^3}+C''(t)\bigg)+\bigg(-\frac{\mathrm{i}}{16t^2}+C'(t)\bigg)^3\bigg]\\
&+\bigg(-48\mathrm{i}+\frac{64}{t}+\frac{\mathrm{i}}{t^2}\bigg)\bigg[\frac{\mathrm{i}}{8t^3}+C''(t)+\bigg(-\frac{\mathrm{i}}{16t^2}+C'(t)\bigg)^2\bigg]\\
&+\bigg(576+\frac{72\mathrm{i}}{t}+\frac{50}{t^2}\bigg)\bigg(-\frac{\mathrm{i}}{16t^2}+C'(t)\bigg)+\bigg(-1728\mathrm{i}+\frac{1008}{t}+\frac{24\mathrm{i}}{t^2}\bigg)\sim0.
\end{align*}
We now note that as $t\rightarrow0$, $C'(t)\ll1/t^2$, $C''(t)\ll1/t^3$, $C^{(3)}(t)\ll1/t^4$ and $(C'(t))^2\ll C'(t)/t^2$. Making these considerations and cancelling terms leads to the asymptotic ODE: as $t\rightarrow0$,
\begin{equation*}\frac{1}{4t^5}-\frac{1}{16t^4}C'(t)\sim0.
\end{equation*}
Thus, $C(t)\sim 4\log(t)$, as $t\rightarrow0$, and we therefore have that
\begin{equation*}\phi_3(t)\sim t^4\mathrm{e}^{\mathrm{i}/16t}, \quad t\rightarrow0.
\end{equation*}
But $\phi_3^{(3)}(t)\sim 16^{-3}t^{-2}\mathrm{e}^{\mathrm{i}/16t}$, as $t\rightarrow0$, and so the third derivative of $\phi_3(t)$ is not well-defined at $t=0$, in contradiction to the finite third absolute moment assumption.

We now focus on the solutions $\phi_1(t)$ and $\phi_2(t)$. It is readily checked that in a small neighbourhood of $t=0$, the solutions have a power series representation with $\phi_1(t)\sim\sum_{k=0}^\infty a_kt^k$ and $\phi_2(t)\sim\sum_{k=1}^\infty b_kt^k$, as $t\rightarrow0$. The general solution is a linear combination of $\phi_1(t)$ and $\phi_2(t)$, and so is a power series of the form $\phi(t)\sim\sum_{k=0}^\infty c_kt^k$, as $t\rightarrow0$, for some other coefficients $\{c_k\}_{k\geq0}$. Under the condition that $\phi(t)$ is the characteristic function of $W$, it follows that $c_k=\mathrm{i}^k\E[W^k]/k!$, $k\geq0$. 

We now argue that the assumption that $\E[W]=0$ uniquely determines the coefficients $\{c_k\}_{k\geq0}$. This would imply that there is a unique solution to (\ref{h4dash}), and that this solution must be the characteristic function  of $H_4(X)$, thus proving sufficiency. 
By assumption, $\E[\mathcal{S}_4'f(W)]=0$ for all $f\in \mathcal{F}_{\mathcal{S}_4',H_4(X)}$. Taking $f(y)=y^k$ then yields the following recurrence relation for the moments of $W$:
\begin{align}\E[W^{k+2}]&=(50+64k+16k(k-1))\E[W^{k+1}]+(24+72k+48k(k-1))\E[W^k]\nonumber\\
\label{recrel}&\quad-(1008k+575k(k-1))\E[W^{k-1}]+1728k(k-1)\E[W^{k-2}], \quad k\geq0.
\end{align}
Setting $k=0$ in (\ref{recrel}) yields the relation $\E[W^2]=50\E[W]+24$, and the assumption $\E[W]=0$ then   gives that $\E[W^2]=24$. With $\E[W]=0$ and $\E[W^2]=24$ we can use forward substitution in (\ref{recrel}) to uniquely determine all moments of $W$ (which we know must be equal to the moments of $H_4(X)$ due to the necessity part of the characterisation). As $c_k=\mathrm{i}^k\E[W^k]/k!$, $k\geq0$, we have succeeded in proving that the sequence $\{c_k\}_{k\geq0}$ is unique, completing the proof of sufficiency.

\vspace{2mm}

\noindent{(v)} The associated characteristic function ODE for $\mathcal{S}_5$ is given by
\begin{align}&(3125t^5+O(t^6))\phi^{(4)}(t)+(31250t^4+O(t^5))\phi^{(3)}(t)\nonumber\\
\label{h5eqn}&+(81875t^3+O(t^4))\phi''(t)+(1+O(t))\phi'(t)+(120t+O(t^2))\phi(t)=0, \quad t\rightarrow0.
\end{align}
We do not present the higher degree coefficients as they will not be needed in our analysis.  This is because we only need to consider the leading behaviour of the solutions in the limit $t\rightarrow0$.  Again, we use the ansatz $\phi(t)=\mathrm{e}^{S(t)}$, and by dominant balance we have that
\begin{equation*}3125(S'(t))^4+\frac{1}{t^5}S'(t)\sim 0, \quad t\rightarrow0.
\end{equation*}
Solving, we have that, as $t\rightarrow0$,
\begin{align*}S_1(t)\sim c_1, \quad S_2(t)\sim \frac{3(-1)^{2/3}}{10(5t)^{2/3}}, \quad S_3(t)\sim \frac{3}{10(5t)^{2/3}}, \quad S_4(t)\sim -\frac{3(-1)^{1/3}}{10(5t)^{2/3}},
\end{align*}
and therefore we have four linearly independent solutions to the characteristic function ODE with the following asymptotic behaviour as $t\rightarrow0$:
\begin{align*}&\phi_1(t)=O(1), \quad \phi_2(t)=\exp\bigg(\frac{3(-1)^{2/3}}{10(5t)^{2/3}}(1+o(1))\bigg), \\
& \phi_3(t)=\exp\bigg(\frac{3}{10(5t)^{2/3}}(1+o(1))\bigg), \quad \phi_4(t)=\exp\bigg(-\frac{3(-1)^{1/3}}{10(5t)^{2/3}}(1+o(1))\bigg).
\end{align*}
We have that $\phi_2(t)\rightarrow\infty$ as $t\rightarrow0^-$ and $\phi_3(t)\rightarrow\infty$ as $t\rightarrow0^+$, contradicting the fact that characteristic functions are bounded.  In contrast, $\phi_4(t)\rightarrow0$, as $t\rightarrow0$.  However, since $(-1)^{1/3}=\frac{1}{2}+\frac{\sqrt{3}}{2}\mathrm{i}$, $\phi_4(t)$ is a complex-valued function of $t$. Now, let $Y=H_5(X)$, and let $\phi_Y(t)$ denote its characteristic function.  As $H_5(x)$ is an odd function it follows that $Y=_d-Y$, meaning that $\phi_Y(t)\in\R$ for all $t\in\R$.  We also know that $\phi_Y(t)$ solves (\ref{h5eqn}) and is linearly independent of $\phi_4(t)$.  Putting this together, we have that a bounded solution of (\ref{h5eqn}) is given by $\phi(t)=A\phi_Y(t)+B\phi_4(t)$, where $A$ and $B$ are arbitrary constants.  In order to meet our assumption that $W$ is a symmetric random variable, we must take $B=0$, and appealing to Lemma \ref{lemcf} completes the proof of sufficiency.


\vspace{2mm}  

\noindent{(vi)} The approach is similar to item (v), but simpler, so we only sketch the details.  The associated characteristic function ODE is
\begin{align*}&(216t^3+O(t^4))\phi^{(4)}(t)+(972t^2+O(t^3))\phi''(t)\\
&+(\mathrm{i}+O(t))\phi'(t)+(720\mathrm{i}t+O(t^2))\phi(t)=0, \quad t\rightarrow0.
\end{align*}
Using the ansatz $\phi(t)=\mathrm{e}^{S(t)}$ and the method of dominant balance leads to the asymptotic ODE
\begin{equation*}216(S'(t))^3+\frac{\mathrm{i}}{t^3}S'(t)\sim 0, \quad t\rightarrow0.
\end{equation*}
Solving in the usual manner then leads to three linearly independent solutions with the following asymptotic behaviour as $t\rightarrow0$:
\begin{align*}&\phi_1(t)=O(1), \quad \phi_2(t)=\exp\bigg(\frac{1-\mathrm{i}}{6\sqrt{3t}}(1+o(1))\bigg), \quad \phi_3(t)=\exp\bigg(-\frac{1-\mathrm{i}}{6\sqrt{3t}}(1+o(1))\bigg).
\end{align*}
We have that $\phi_2(t)\rightarrow\infty$, as $t\rightarrow0^+$, and $\phi_2(t)\rightarrow\infty$, as $t\rightarrow0^-$, and so appealing to Lemma \ref{lemcf} gives us sufficiency.

\vspace{2mm}  

\noindent{(vii)} The approach is very similar to item (vi) for $H_5(X)$.  We therefore only sketch the details.  The associated characteristic function ODE is
\begin{align*}&(823543t^7+O(t^8))\phi^{(6)}(t)+(17294403t^6+O(t^7))\phi^{(5)}(t)+(116825457t^5+O(t^6))\phi^{(4)}(t)\\&+(306156312t^4+O(t^5))\phi^{(3)}(t)+(306955845t^3+O(t^4))\phi''(t)+(1+O(t))\phi'(t)\\
&+(5040t+O(t^2))\phi(t)=0, \quad t\rightarrow0.
\end{align*}
Using the ansatz $\phi(t)=\mathrm{e}^{S(t)}$ and the method of dominant balance leads to the asymptotic ODE
\begin{equation*}823543(S'(t))^6+\frac{1}{t^7}S'(t)\sim 0, \quad t\rightarrow0.
\end{equation*}
Solving in the usual manner then leads to six linearly independent solutions with the following asymptotic behaviour as $t\rightarrow0$:
\begin{align*}&\phi_1(t)=O(1), \quad \phi_2(t)=\exp\bigg(\frac{5}{14(7t)^{2/5}}(1+o(1))\bigg), \\
& \phi_3(t)=\exp\bigg(-\frac{5(-1)^{1/5}}{14(7t)^{2/5}}(1+o(1))\bigg), \quad \phi_4(t)=\exp\bigg(\frac{5(-1)^{2/5}}{14(7t)^{2/5}}(1+o(1))\bigg),
\\
& \phi_5(t)=\exp\bigg(-\frac{5(-1)^{3/5}}{14(7t)^{2/5}}(1+o(1))\bigg), \quad \phi_6(t)=\exp\bigg(\frac{5(-1)^{4/5}}{14(7t)^{2/5}}(1+o(1))\bigg).
\end{align*}
The solutions $\phi_k(t)$, $k=2,4,5,6$, blow up in either the limits $t\rightarrow0^-$ or $t\rightarrow0^+$ (or both), whilst $\phi_3(t)\rightarrow0$ as $t\rightarrow0$.  But $\phi_3(t)$ is a complex-valued function of $t$, and arguing as in part (iv) gives us our proof of sufficiency. 

\vspace{2mm}  

\noindent{(viii)} The argument begins similarly to in item (vi), but the analysis is a little more involved.  We sketch the first part of the analysis that is similar to item (vi). The associated characteristic function ODE is
\begin{align*}&(4096t^4+O(t^5))\phi^{(4)}(t)+(32768t^3+O(t^4))\phi^{(3)}(t)\\
&+(65920t^2+O(t^3))\phi''(t)+(\mathrm{i}+O(t))\phi'(t)+(40320\mathrm{i}t+O(t^2))\phi(t)=0, \quad t\rightarrow0.
\end{align*}
Whilst a more refined analysis is needed here, the higher degree coefficients will still not appear in our analysis, so we again do not present them. Using the ansatz $\phi(t)=\mathrm{e}^{S(t)}$ and the method of dominant balance leads to the asymptotic ODE
\begin{equation*}4096(S'(t))^4+\frac{\mathrm{i}}{t^4}S'(t)\sim 0, \quad t\rightarrow0.
\end{equation*}
Solving in the usual manner then leads to four linearly independent solutions with the following asymptotic behaviour as $t\rightarrow0$:
\begin{align*}&\phi_1(t)=O(1), \quad \phi_2(t)=\exp\bigg(-\frac{3\mathrm{i}}{16t^{1/3}}(1+o(1))\bigg), \\
& \phi_3(t)=\exp\bigg(\frac{3(-1)^{1/6}}{16t^{1/3}}(1+o(1))\bigg), \quad \phi_4(t)=\exp\bigg(\frac{3(-1)^{5/6}}{16t^{1/3}}(1+o(1))\bigg).
\end{align*}
The solution $\phi_3(t)$ blows up as $t\rightarrow0$, but more analysis is needed for the solutions $\phi_2(t)$ and $\phi_4(t)$.  We proceed as in the proof of part (ii) of Theorem \ref{prop2ndode} and seek the refined approximations
\begin{align*}S_2(t)=-\frac{3\mathrm{i}}{16t^{1/3}}+C_2(t), \quad S_4(t)=\frac{3(-1)^{5/6}}{16t^{1/3}}+C_4(t),
\end{align*}
where, for $k=2,4$, $C_k(t)\ll 1/t^{1/3}$, as $t\rightarrow0$.  The method of analysis used to obtain asymptotic approximations for the correction terms $C_2(t)$ and $C_4(t)$ is similar to the one in the proof of part (ii) of Theorem \ref{prop2ndode}, and as such our analysis will not be as detailed. We first note that
\begin{align*}\phi_2'(t)&=\bigg(\frac{\mathrm{i}}{16t^{4/3}}+C_2'(t)\bigg)\mathrm{e}^{S_2(t)}, \\
\phi_2''(t)&=O(t^{-8/3})\times\mathrm{e}^{S_2(t)}, \quad t\rightarrow0, \\
\phi_2^{(3)}(t)&=O(t^{-4})\times \mathrm{e}^{S_2(t)}, \quad t\rightarrow0, 
\end{align*}
and
\begin{align*}
\phi_2^{(4)}(t)&=\bigg[\frac{-35\mathrm{i}}{54t^{13/3}}+C_2^{(4)}(t)+3\bigg(-\frac{\mathrm{i}}{12t^{7/3}}+C_2''(t)\bigg)^2\\&\quad+4\bigg(\frac{\mathrm{i}}{16t^{4/3}}+C_2'(t)\bigg)\bigg(\frac{7\mathrm{i}}{36t^{10/3}}+C_2^{(3)}(t)\bigg)\\
&\quad+6\bigg(\frac{\mathrm{i}}{16t^{4/3}}+C_2'(t)\bigg)^2\bigg(-\frac{7\mathrm{i}}{12t^{7/3}}+C_2''(t)\bigg)+\bigg(\frac{\mathrm{i}}{16t^{4/3}}+C_2'(t)\bigg)^4\bigg] \mathrm{e}^{S_2(t)}.
\end{align*}
Using the method of dominant balance and our assumption that $C_2(t)\ll 1/t^{1/3}$, $t\rightarrow0$, (we will soon see that our assumption is met) we find that $C_2(t)$ satisfies the asymptotic ODE
\begin{align*}4096\times6\times\bigg(\frac{\mathrm{i}}{16}\bigg)^2\times \bigg(-\frac{\mathrm{i}}{12}\bigg) \times \frac{1}{t^5}+\bigg[4096\times4\times\bigg(\frac{\mathrm{i}}{16}\bigg)^3+\mathrm{i}\bigg]\frac{C_2'(t)}{t^4}\sim0, \quad t\rightarrow0,
\end{align*} 
that is
\[\frac{8}{t^5}-\frac{3}{t^4}C_2'(t)\sim 0, \quad t\rightarrow0,\]
which has solution $C_2(t)\sim\frac{8}{3}\log(t)$, as $t\rightarrow0$. We therefore have that
\[\phi_2(t)\sim t^{8/3}\mathrm{e}^{-3\mathrm{i}/16t^{1/3}}, \quad t\rightarrow0.\]
But $\phi_2''(t)\sim-\frac{1}{256}\mathrm{e}^{-3\mathrm{i}/16t^{1/3}}$, as $t\rightarrow0$, and so the second derivative of $\phi_2(t)$ is not well-defined at $t=0$, in contradiction to the finite fourth moment assumption (which implies a finite second moment assumption).  The analysis for $\phi_4(t)$ is almost identical and we have $\phi_4(t)\sim t^{8/3}\mathrm{e}^{3(-1)^{5/6}/16t^{1/3}}$, as $t\rightarrow0$,
which again does not have a well-defined second derivative at $t=0$. Finally, appealing to Lemma \ref{lemcf} gives us sufficiency and completes the proof.
\end{proof}

\begin{rem}\label{remcharcc}

\begin{itemize}

\item [(i)] In Proposition \ref{prophpchar}, we showed that the Stein operators for $H_p(X)$, $p=3,\ldots,8$, with minimum possible maximal polynomial degree $m$ are characterising.  Our proof involved carrying out an asymptotic analysis of the behaviour of the solutions of the associated characteristic function ODEs in the limit $t\rightarrow0$ on a case-by-case basis.  By carrying out similar, but increasingly complex, analyses it may be possible to prove that the Stein operators for $H_p(X)$  with minimum possible maximal polynomial degree $m$ are characterising for values of $p$ greater than 8.  However, as this is carried out on a case-by-case basis via involved calculations there is a place at which one must stop ; see also Remark \ref{pnrem} for a discussion as to why the case $p=9$ seems to more difficult than $p\leq8$.  Nevertheless, by establishing that the Stein operator for $H_3(X)$ with $(T,m)=(5,2)$ was characterising, the important `threshold' of moment indeterminacy was passed ($H_p(X)$ is determined by its moments for $p=1,2,4$, but not for $p=3$ and $p\geq5$).  Moreover, by then confirming that the Stein operators for $H_p(X)$, $p=4,\ldots,8$, with minimum possible maximal polynomial degree $m$ are characterising, we believe it is reasonable to conjecture that this is the case for all $p\geq1$. Similar comments apply for Proposition \ref{propnnn} below concerning Stein characterisations related to the product of $p\geq1$ independent standard Gaussian random variables. 

\item [(ii)] \rg{In part (ii) of Proposition \ref{prophpchar}, our characterisation of $H_4(X)$ has an additional moment condition ($\E[W]=0$), whilst in parts (vi) and (viii) our characterisations of $H_6(X)$ and $H_8(X)$ had an additional symmetry condition ($W=_d-W$). We needed to impose these additional conditions in order to argue through our asymptotic approach that there is a unique solution to the ODEs that appear in the proof that has the properties of characteristic functions. It is certainly possible that as the Stein operators for $H_p(X)$ become more complicated as $p$ increases that additional conditions are required to ensure characterisation. However, our intuition is that these additional conditions are not necessary to ensure characterisation, but are rather artefacts resulting from our proof. Similar comments apply for Proposition \ref{propnnn} below in which some of our characterisations of the product of $p\geq1$ independent standard Gaussian random variables involve an additional symmetry condition.}

\item [(iii)] For polynomial Stein operators with coefficients of degree $m\geq3$, our asymptotic proof technique has a lower success rate than the degree 2 case.  As an illustrative example, consider the case $m=3$.  Here it may be the case that an asymptotic analysis in the limits $t\rightarrow0$ and $t\rightarrow\infty$ shows that there are three linearly independent solutions with one unbounded solution and two that are consistent with a characteristic function in both these limits.  However, from this analysis alone, we have no way of knowing whether we have two unbounded solutions (one unbounded at $t=0$ and one unbounded a $t\rightarrow\infty$) or just one unbounded solution (unbounded at both $t=0$ and as $t\rightarrow\infty$).  On this basis, we were fortunate that just knowledge of the asymptotic behaviour of the solutions of the associated characteristic function ODEs in the limit $t\rightarrow0$ was enough to establish sufficiency for each of the Stein operators of Proposition \ref{prophpchar}.  We cannot guarantee that this is the case for other Stein operators for Gaussian polynomials, and expect that the asymptotic proof techniques of this section will often break down for Stein operators with $m\geq3$.

\end{itemize}
\end{rem}




\subsection{The product of independent standard Gaussian random variables}\label{sec3.3}

Let $X_1,\ldots,X_p$ be independent with $X_k\sim N(0,\sigma_k^2)$, $k=1,\ldots,p$. Let $\sigma^2=\prod_{k=1}^p\sigma_k^2$, and let $Y_p=\prod_{k=1}^pX_k$. We write $Y_p\sim\mathrm{PN}(p,\sigma^2)$. The $\mathrm{PN}(p,\sigma^2)$ Stein operator of \cite{gaunt-pn} is given by
\begin{align}\mathcal{S}_pf(y)&=\frac{\sigma^2}{y}\bigg(y\frac{\mathrm{d}}{\mathrm{d}y}\bigg)^pf(y)-yf(y)\nonumber\\
\label{pnsteinop}&=\sigma^2\sum_{k=1}^p{p\brace k}y^{k-1}f^{(k)}(y)-yf(y),
\end{align}
where ${p\brace k}=\frac{1}{k!}\sum_{j=0}^k(-1)^{k-j}\binom{k}{j}j^p$ is a Stirling number of the second kind  (see \cite{Olver}). We note that ${p\brace 1}={p\brace p}=1$ for all $p\geq1$. It is known that the Stein operators $\mathcal{S}_1$ and $\mathcal{S}_2$ are characterising; the following proposition provides Stein characterisations for $\mathcal{S}_p$, $p=3,\ldots,8$.

\begin{prop}\label{propnnn}Let $W$ be a real-valued random variable. Let $Y_p\sim\mathrm{PN}(p,\sigma^2)$ and let $\mathcal{S}_p$ denote the Stein operator (\ref{pnsteinop}).  Then
\begin{itemize}

\item[(a)] Let $p=3,4$. Suppose $\E[|W|^{p-1}]<\infty$. Then $W\sim \mathrm{PN}(p,\sigma^2)$ if and only if $\E  [\mathcal{S}_pf(W)]=0$ for all $f\in\mathcal{F}_{\mathcal{S}_p,Y_p}$.

\item[(b)]  Let $p=5,6,7,8$. Suppose $W$ is a symmetric random variable such that $\E[|W|^{p-1}]<\infty$. Then $W\sim \mathrm{PN}(p,\sigma^2)$ if and only if $\E  [\mathcal{S}_pf(W)]=0$ for all $f\in\mathcal{F}_{\mathcal{S}_p,Y_p}$.

\end{itemize}
\end{prop}

\begin{proof}Throughout this proof, we set $\sigma^2=1$; the case $\sigma^2>0$ follows easily by rescaling. We begin by noting that the associated characteristic function ODE for $\mathcal{S}_p$ is given by 
\begin{equation}\label{pncfeqn}\sum_{k=1}^{p}{p\brace k}t^k\phi^{(k-1)}(t)+\phi'(t)=0.
\end{equation}
Using the ansatz $\phi(t)=\mathrm{e}^{S(t)}$ and the method of dominant balance yields the asymptotic ODE
\begin{equation*}(S'(t))^{p-1}+\frac{S'(t)}{t^p}\sim 0, \quad t\rightarrow0.
\end{equation*}

\vspace{2mm}

\noindent{(i)} It was shown on p.\ 3320 of \cite{gaunt-pn} that, in the case $p=3$, there is unique solution to (\ref{pncfeqn}) satisfying the conditions of a characteristic function (with an explicit formula given). Sufficiency therefore follows from Lemma \ref{lemcf}.

\vspace{2mm}

\noindent{(ii)} Solving in the usual manner gives that, for $p=4$,
\begin{align*}\phi_1(t)=O(1), \quad \phi_2(t)=\exp\bigg(-\frac{1}{t}(1+o(1))\bigg), \quad \phi_3(t)=\exp\bigg(\frac{1}{t}(1+o(1))\bigg).
\end{align*}
The solution $\phi_2(t)$ blows up as $t\rightarrow0^-$ and the solution $\phi_3(t)$ blows up as $t\rightarrow0^+$, proving sufficiency.

\vspace{2mm}

\noindent{(iii)} For $p=5$,
\begin{align*}&\phi_1(t)=O(1), \quad \phi_2(t)=\exp\bigg(\frac{3}{2t^{2/3}}(1+o(1))\bigg), \\
& \phi_3(t)=\exp\bigg(-\frac{1+\sqrt{3}\mathrm{i}}{4t^{2/3}}(1+o(1))\bigg), \quad \phi_4(t)=\exp\bigg(-\frac{1-\sqrt{3}\mathrm{i}}{4t^{2/3}}(1+o(1))\bigg).
\end{align*}
We have that $\phi_2(t)\rightarrow\infty$, as $t\rightarrow0^+$, and $\phi_3(t)\rightarrow\infty$, as $t\rightarrow0^-$. The solution $\phi_4(t)$ does not blow up as $t\rightarrow0$, but it is a complex-valued. function. Arguing as we did in part (v) of Proposition \ref{prophpchar} then gives us sufficiency.

\vspace{2mm}

\noindent{(iv)} For $p=6$,
\begin{align*}&\phi_1(t)=O(1), \quad \phi_2(t)=\exp\bigg(\frac{\sqrt{2}(1+\mathrm{i})}{t^{1/2}}(1+o(1))\bigg), \\
 &\phi_3(t)=\exp\bigg(-\frac{\sqrt{2}(1+\mathrm{i})}{t^{1/2}}(1+o(1))\bigg), \quad \phi_4(t)=\exp\bigg(-\frac{\sqrt{2}(1-\mathrm{i})}{t^{1/2}}(1+o(1))\bigg), \\
&  \phi_5(t)=\exp\bigg(\frac{\sqrt{2}(1-\mathrm{i})}{t^{1/2}}(1+o(1))\bigg).
\end{align*}
We have that $\phi_2(t)\rightarrow\infty$, as $t\rightarrow0^+$, $\phi_4(t)\rightarrow\infty$, as $t\rightarrow0^-$, and $\phi_5(t)\rightarrow\infty$, as $t\rightarrow0^+$. In contrast $\phi_3(t)\rightarrow0$, as $t\rightarrow0$. But $\phi_3(t)$ is complex-valued, so arguing as we did in part (v) of Proposition \ref{prophpchar} gives us sufficiency.

\vspace{2mm}

\noindent{(v)} For $p=7$,
\begin{align*}&\phi_1(t)=O(1), \quad \phi_2(t)=\exp\bigg(\frac{5}{2t^{2/5}}(1+o(1))\bigg),  \\
& \phi_3(t)=\exp\bigg(-\frac{5(-1)^{1/5}}{2t^{2/5}}(1+o(1))\bigg), \quad \phi_4(t)=\exp\bigg(\frac{5(-1)^{2/5}}{2t^{2/5}}(1+o(1))\bigg), \\
& \phi_5(t)=\exp\bigg(-\frac{5(-1)^{3/5}}{2t^{2/5}}(1+o(1))\bigg), \quad \phi_6(t)=\exp\bigg(\frac{5(-1)^{4/5}}{2t^{2/5}}(1+o(1))\bigg).
\end{align*}
We have that $\phi_2(t)\rightarrow\infty$, as $t\rightarrow0^+$, $\phi_4(t)\rightarrow\infty$, as $t\rightarrow0^+$, $\phi_5(t)\rightarrow\infty$, as $t\rightarrow0^+$, and $\phi_6(t)\rightarrow\infty$, as $t\rightarrow0^-$. In contrast, $\phi_3(t)\rightarrow0$, as $t\rightarrow0$, but as $\phi_3(t)$ is complex-valued, the usual argument gives us sufficiency.

\vspace{2mm}

\noindent{(vi)} For $p=8$,
\begin{align*}&\phi_1(t)=O(1), \quad \phi_2(t)=\exp\bigg(\frac{3\mathrm{i}}{t^{1/3}}(1+o(1))\bigg), \quad \phi_3(t)=\exp\bigg(-\frac{3\mathrm{i}}{t^{1/3}}(1+o(1))\bigg), \\
& \phi_4(t)=\exp\bigg(\frac{3(-1)^{1/6}}{t^{1/3}}(1+o(1))\bigg), \quad \phi_5(t)=\exp\bigg(-\frac{3(-1)^{1/6}}{t^{1/3}}(1+o(1))\bigg), \\
& \phi_6(t)=\exp\bigg(\frac{3(-1)^{5/6}}{t^{1/3}}(1+o(1))\bigg), \quad \phi_7(t)=\exp\bigg(-\frac{3(-1)^{5/6}}{t^{1/3}}(1+o(1))\bigg).
\end{align*}
We have that $\phi_2(t)\rightarrow\infty$, as $t\rightarrow0^-$, $\phi_4(t)\rightarrow\infty$, as $t\rightarrow0^+$, and $\phi_7(t)\rightarrow\infty$, as $t\rightarrow0^-$. In contrast, $\phi_5(t)\rightarrow0$, as $t\rightarrow0$, although we do note that $\phi_5(t)$ is complex-valued. For, $\phi_3(t)$ and $\phi_6(t)$, we need to obtain a more precise description of their asymptotic behaviour as $t\rightarrow0$ to be able to conclude whether they blow up in this limit. Performing a similar analysis to the one used in the proof of part (viii) of Proposition \ref{prophpchar} gives the refined asymptotic approximations: 
\begin{equation*}\phi_3(t)\sim t^{-33/2}\mathrm{e}^{-3\mathrm{i}/t^{1/3}}, \quad \phi_6(t)\sim t^{-33/2}\mathrm{e}^{3(-1)^{5/6}/t^{1/3}}, \quad t\rightarrow0.
\end{equation*}
 We will not provide such a detailed analysis as the one used in the proof of part (viii) of Proposition \ref{prophpchar}, and will instead just confirm that the exponent $-33/2$ is the correct one. We show this for the solution $\phi_3(t)$, with the a similar analysis applying for $\phi_6(t)$. 

Suppose that $\phi_3(t)\sim t^{a}\mathrm{e}^{-3\mathrm{i}/t^{1/3}}$, as $t\rightarrow0$; we aim to show that $a=-33/2$. Then, as $t\rightarrow0$, 
\begin{align*}t^8\phi_3^{(7)}(t)&\sim(-\mathrm{i}t^{a-4/3}+(28-7a)t^{a-1})\mathrm{e}^{-3\mathrm{i}/t^{1/3}}, \\
t^7\phi_3^{(6)}(t)&\sim-t^{a-1}\mathrm{e}^{-3\mathrm{i}/t^{1/3}}, \\
\phi_3'(t)&\sim(\mathrm{i}t^{a-4/3}+at^{a-1})\mathrm{e}^{-3\mathrm{i}/t^{1/3}}, \\
t^k\phi_3^{(k-1)}(t)&=o(t^{a-1}\mathrm{e}^{-3\mathrm{i}/t^{1/3}}), \quad k=1,2,3,4,5. 
\end{align*}
Substituting these limiting forms into (\ref{pncfeqn}), using that ${8\brace 1}={8\brace 8}=1$, ${8\brace 2}=127$, canceling the $O(t^{a-4/3}\mathrm{e}^{-3\mathrm{i}/t^{1/3}})$ terms, and equating the $O(t^{a-1}\mathrm{e}^{-3\mathrm{i}/t^{1/3}})$ terms yields that the exponent $a$ satisfies the equation $(28-7a)+127(-1)+a=0$, so that $a=-33/2$, as required.


To conclude the proof, we now observe from the refined approximations for $\phi_3(t)$ and $\phi_6(t)$ that $\phi_3(t)\rightarrow\infty$, as $t\rightarrow0^+$, and  $\phi_6(t)\rightarrow\infty$, as $t\rightarrow0^-$. We thus deduce sufficiency by the usual argument.
\end{proof}

\begin{rem}\label{pnrem} We now demonstrate how the analysis becomes more challenging in the case $p=9$; the analysis for the Gaussian Hermite polynomial $H_9(X)$ also becomes more demanding for the same reason. For $p=9$,
\begin{align*}&\phi_1(t)=O(1), \quad \phi_2(t)=\exp\bigg(\frac{7}{2t^{2/7}}(1+o(1))\bigg), \\ 
&\phi_3(t)=\exp\bigg(-\frac{7(-1)^{1/7}}{2t^{2/7}}(1+o(1))\bigg), \quad \phi_4(t)=\exp\bigg(\frac{7(-1)^{2/7}}{2t^{2/7}}(1+o(1))\bigg), \\
& \phi_5(t)=\exp\bigg(-\frac{7(-1)^{3/7}}{2t^{2/7}}(1+o(1))\bigg), \quad \phi_6(t)=\exp\bigg(\frac{7(-1)^{4/7}}{2t^{2/7}}(1+o(1))\bigg), \\ &\phi_7(t)=\exp\bigg(-\frac{7(-1)^{5/7}}{2t^{2/7}}(1+o(1))\bigg),\quad \phi_8(t)=\exp\bigg(\frac{7(-1)^{6/7}}{2t^{2/7}}(1+o(1))\bigg).
\end{align*}
The solutions $\phi_2(t)$, $\phi_4(t)$, $\phi_6(t)$ and $\phi_7(t)$ are unbounded as $t\rightarrow0^-$ or $t\rightarrow0^+$; however, $\phi_3(t)$, $\phi_5(t)$ and $\phi_8(t)$ tend to zero as $t\rightarrow0$. These solutions  are complex-valued, but this is not enough to deduce sufficiency because it is possible to take linear combinations of the solutions that are real-valued to leading order. Indeed,
\begin{align*}\phi_3(t)+\phi_8(t)=2\exp\bigg(-\frac{7\cos(\pi/7)}{2t^{2/7}}(1+o(1))\bigg)\cos\bigg(\frac{7\sin(\pi/7)}{2t^{2/7}}(1+o(1))\bigg).
\end{align*}
Moreover, $\phi_3(t)+\phi_8(t)$ is infinitely differentiable as $t\rightarrow0$ and $\phi_3^{(n)}(0)+\phi_8^{(n)}(0)=0$ for all $n\geq0$, so we cannot use moment conditions to deduce sufficiency. A more detailed analysis may allow one to deduce sufficiency; for example, finding additional terms in the asymptotic expansions of $\phi_3(t)$ and $\phi_8(t)$ may show that even though $\phi_3(t)+\phi_8(t)$ is real-valued at leading order the function itself is not real-valued. Whether or not such a strategy would succeed, this example demonstrates how the analysis becomes more challenging for $p\geq9$.
\end{rem}

\section{Gamma characterisations in Malliavin calculus}\label{sec4}

We refer the reader to the excellent textbook \cite[Chapter 2]{n-p-book} for any unexplained notion evoked in this section. Let $W = \{W(h) : \HH\}$ be an isonormal Gaussian process over some real separable Hilbert space $\HH$, with inner product $\langle \cdot,\cdot\rangle _{ \HH}$. This means that $W$ is a centered Gaussian family, defined on some probability space $(\Omega ,\mathcal{F},P)$, with a covariance structure given by the relation
$\E\left[ W(h)W(g)\right] =\langle h,g\rangle _{ \HH}$. We also assume that $\mathcal{F}=\sigma(W)$, that is, $\mathcal{F}$ is generated by $W$, and we use the shorthand notation $L^2(\Omega) := L^2(\Omega, \mathcal{F}, P)$. For every $q\geq 1$, the symbol $C_{q}$ stands for the $q$-th {\it Wiener chaos} of $W$, defined as the closed linear subspace of $L^2(\Omega)$
generated by the family $\{H_{q}(W(h)) : h\in  \HH,\left\| h\right\| _{ \HH}=1\}$, where $H_{q}$ is the $q$-th Hermite polynomial.
We write by convention $C_{0} = \R$. It is well-known that $L^2(\Omega)$ can be decomposed into the infinite orthogonal sum of the spaces $C_{q}$: this means that any square-integrable random variable
$F\in L^2(\Omega)$ admits the {\it Wiener-It\^{o} chaotic expansion} $F=\sum_{q=0}^{\infty }J_q (F)$,  
where the series converges in $L^2(\Omega)$, $J_0 (F)=\E[F]$, and the projections $J_q(F)$ are uniquely determined by $F$. 

In what follows, for the sake of simplicity, we assume that all the random elements belong to a finite sum of Wiener chaoses. Next, we recall the Gamma operators of Malliavin calculus. 
We let $D$ and $L$ stand for the Malliavin derivative and the Ornstein-Uhlenbeck generator, respectively.  We also let $L^{-1}$ denote the pseudoinverse of $L$.
 
\begin{mydef}\label{Def : Gamma}
	Let $F$ be a random variable that belongs to a finite sum of Wiener chaoses.  The sequence of random variables $\{\Gamma_r(F)\}_{r\geq 0}$ is recursively defined as follows. Set $\Gamma_0(F) = F$
		and, for every $r\geq 1$,
		\begin{align*} 
		\Gamma_{r}(F) = \langle DF,-DL^{-1}\Gamma_{r-1}(F)\rangle_{\HH}.
		\end{align*}
\end{mydef}



In the proof of the following proposition, we shall need the following formula (see Lemma 4.2 and Theorem 4.3 of \cite{np10}): Suppose $F$ belongs to a finite sum of Wiener chaoses, then, for $r\geq0$,
\begin{equation*}r!\E[\Gamma_r(F)]=\kappa_{r+1}(F),
\end{equation*}
where $\kappa_{r+1}(F)$ is the $(r+1)$-th cumulant of $F$. We shall also make repeated use of the following Malliavin integration by parts formula (see \cite[Theorem 2.9.1]{n-p-book}).  Let $F,G$ be random variables belonging to a finite sum of Wiener chaoses, and let $g$ have a bounded derivative.  Then
\begin{equation}\label{mibp}\E[Fg(G)]=\E[F]\E[g(G)]+\E\big[g'(G)\langle DG,-DL^{-1}F\rangle_{\HH}\big].
\end{equation}

\begin{prop}\label{prop:Gamma-Type-Expressions}
	Let $X\sim N(0,1)$. Let $Y$ be a centered
	random element belonging to a finite sum of Wiener chaoses. Then the following "Gamma characterisations" hold:
	\begin{itemize}
		\item[(a)] For $Y = H_3 (X)$ we have
		\begin{equation}\label{eq:Gamma3}
		\Gamma_5(Y) -153\Gamma_3(Y) -27 Y \Gamma_2 (Y) + 324 \Gamma_1 (Y) - 486 (4-Y^2) =0, \, a.s. 
		\end{equation}
		Conversely, if $\kappa_r(Y) = \kappa_r(H_3(X))$ for $r=2,3,4,5$, and in addition relation \eqref{eq:Gamma3} holds then $Y =_d H_3 (X)$. 
		\item[(b)] For $Y = H_3 (X)$ we have
		\begin{equation}\label{eq:Gamma3-1ess-cumulants}
		\Gamma_4(Y) +3 Y\Gamma_3(Y) - 540 \Gamma_2 (Y) - 351 Y \Gamma_1 (Y) + 81 Y (4-Y^2) =0, \, a.s. 
		\end{equation}
		Conversely, if $\kappa_r(Y) = \kappa_r(H_3(X))$ for $r=2,3,4$, and \eqref{eq:Gamma3-1ess-cumulants} holds then $Y =_d H_3 (X)$.
		\item[(c)] For $Y = H_4 (X)$ we have  
		\begin{equation}\label{eq:Gamma4}
		\Gamma_3 (Y) - 60 \Gamma_2(Y) + 16 (9-Y) \Gamma_1(Y) - 192 (Y+6)(3-Y) = 0,\, a.s.
		\end{equation}
		Conversely, if $\kappa_r(Y) = \kappa_r(H_4(X))$ for $r=2,3$, and \eqref{eq:Gamma4} holds then $Y=_d H_4 (X)$.
	\end{itemize}
\end{prop}
\begin{proof}
We prove item (a); the others are similar.	First assume that $Y= H_3 (X)$. Proposition \ref{prophpchar}, item (iii) implies that 
	\begin{align}
	\E \big[   486 ( 4- Y^2) f ^{(5)} (Y) - 486 Y f^{(4)} (Y) -27 (8-Y^2) f^{(3)}(Y) &\nonumber \\
\label{eq:H3-Stein-Identity}	+ 99 Y f''(Y)+ 6 f'(Y) - Y f(Y)  \big]=0.&
	\end{align}
 Next, we use several times the Malliavin integration by parts formula (\ref{mibp}) to write down all the expressions in \eqref{eq:H3-Stein-Identity} of expectations involving the lower derivatives of $f$ in terms of expressions having only fifth derivative of $f$. At this point, relying on a standard regularisation argument, by convoluting $f$ by an approximation of the identity, we can assume that the first five derivatives of $f$ are bounded. Now, since $\E[Y]=0$, we have that $ \E[Y f(Y)] = \E[ f'(Y) \Gamma_1 (Y)]$. 	We also have $ \E\left[ f' (Y) \left( 6 - \Gamma_1 (Y)\right)\right] = - \E \left[ f''(Y) \Gamma_2 (Y)\right]$, since $\E \left[ \Gamma_1 (Y)\right] = \E \left[ Y^2\right]=6$. Therefore, taking into account that $ \E \left[99 Y - \Gamma_2 (Y) \right] = - \frac{1}{2} \kappa_3 (Y)=0$, we obtain 
 $$\E \big[ f''(Y) \left( 99 Y - \Gamma_2 (Y) \right)\big] = \E \big[ f^{(3)}(Y)   \left( 99 \Gamma_1 (Y) - \Gamma_3 (Y)\right)\big].$$ 
 Hence, using Malliavin integration by parts once more we arrive at
	\begin{align}
	&\E \big[  f^{(3)} (Y) \left( -27 (8 - Y^2) + 99 \Gamma_1(Y) - \Gamma_3 (Y) \right) \big]\nonumber\\
\label{eq:non-linear-structure}	&= \E \big[ f^{(4)}(Y) \left(  99\Gamma_2 (Y) - \Gamma_4 (Y) + 27 \langle DY, - DL^{-1} Y^2 \rangle  \right)  \big],
	\end{align}
	where we used that $\E [   27 (8 - Y^2) - 99 \Gamma_1(Y) + \Gamma_3 (Y)  ]= \frac{1}{3!}\kappa_4 (Y) - 126 \kappa_2 (Y) = 0$. Due to presence of non-linearity in the expression $\langle DY, - DL^{-1} Y^2 \rangle_{\HH} $, at this point, one cannot directly apply Malliavin integration by parts to proceed. To overcome this little issue, we note that for any random variable $F$ in the $p$-th Wiener chaos it holds that $ L^{-1}(F^2) =  L^{-1} \Gamma_1 (F) - \frac{1}{2p} \left( F^2 - \E[F^2] \right)$. To see this, we note the pseudoinverse property $L^{-1} L F^2 = F^2 - \E[F^2]$. On the other hand, by definition of the carr\'e du champ operator $2\Gamma[F,G]=L(FG) - F LG - G LF$, we can write 
\begin{align*}
2\Gamma[F,F] = LF^2 - 2 F L F = LF^2 + 2p F^2.  
\end{align*}
Applying $L^{-1}$ to both sides, and using the fact that on the Wiener space $\Gamma[F,G]= \langle DF,DG \rangle_{\HH}$, we obtain 
\begin{align*}
2 L^{-1}  \Gamma[F,F] = 2p L^{-1} \Gamma_1 (F) = F^2 - \E[F^2] + 2p L^{-1} F^2,
\end{align*}
as required. Therefore
	\[ \langle DY, - DL^{-1} Y^2 \rangle_{\HH} = \Gamma_2 (Y) + Y \Gamma_1 (Y).  \] Plugging into \eqref{eq:non-linear-structure}, we obtain that 
	\begin{align*}
	\E & \big[  f^{(4)} (Y)\left(  -486 Y - \Gamma_4(Y) + 126\Gamma_2(Y) + 27 Y \Gamma_1 (Y) \right) \big] \\
	 &= \E \big[  f^{(4)} (Y)\left(    - \Gamma_4(Y) + 126\Gamma_2(Y)  \right) \big] +  27 \E \big[  f^{(4)} (Y) Y \left( -18 + \Gamma_1(Y) \right) \big]\\
	&= \E \big[  f^{(4)} (Y)\left(    - \Gamma_4(Y) + 126\Gamma_2(Y)  \right) \big] + 
	27 \E \big[  f^{(4)} (Y) Y \left( -6 + \Gamma_1(Y) \right) \big]\\
	&\quad - 27\cdot12 \E \big[ f^{(4)} (Y) Y \big]\\
	&:= I_1 + 27 I_2 - 12\cdot27I_3.
	\end{align*}
	Next, we have $ I_1 = \E [   f^{(5)} (Y)    \left(    - \Gamma_5(Y) + 126\Gamma_3(Y)  \right)  ]$ (since $\kappa_3(Y)=\kappa_5(Y)=0$) and $I_3 = \E [  f^{(5)} (Y) \Gamma_1 (Y) ]$. Moreover,
	\begin{align*}
	I_2 = \E\big[ \langle  D (Y  f^{(4)} (Y)) , - DL^{-1} \Gamma_1(Y) \big]
	&= \E \big[  f^{(5)} (Y) Y \Gamma_2 (Y) \big] + \E \big[  f^{(4)} (Y) \Gamma_2 (Y) \big]\\
	&=  \E \big[  f^{(5)} (Y) Y \Gamma_2 (Y) \big] + \E \big[  f^{(5)} (Y) \Gamma_3 (Y) \big].
	\end{align*}
	Gathering all the terms, we obtain,  for $f:\R \to \R$, that
	\[ \E \Bigg[  f (Y)  \Bigg(   \underbrace{\Gamma_5(Y) -153\Gamma_3(Y) -27 Y \Gamma_2 (Y) + 324 \Gamma_1 (Y) - 486 (4-Y^2)}_{:=A(Y)}  \Bigg)    \Bigg] =0. \] The latter implies that $\E\left[ A(Y)   \,\vert\, Y \right]=0$ and that completes the proof because the associated polynomial Stein operator is algebraic. For the other direction, having $\kappa_r(Y)=\kappa_r(H_3(X))$ for $r=2,3,4,5$, and proceeding backwards in the above steps we arrive to 
	\begin{align}
	\E\big[486(4-Y^2)f^{(5)}(Y)-486Yf^{(4)}(Y)-27(8-Y^2)f^{(3)}(Y)&\nonumber\\
+99Yf''(Y)+6f'(Y)-Yf(Y)\big]=0&.\nonumber
\end{align}
Now, we use Proposition \ref{prophpchar}, item (iii). 
\end{proof}

\begin{rem}\label{sec4rem}
\begin{itemize} \item[(i)] The Gamma expressions \eqref{eq:Gamma3}, \eqref{eq:Gamma3-1ess-cumulants} and \eqref{eq:Gamma4} differ in two major features compared with their counterparts when the target random variable $Y$ belongs to the second Wiener chaos. For a typical element in the second Wiener chaos with a finite spectral decomposition, a linear combination of Gamma operators with constant coefficients coincide with a polynomial in the target of degree one, see \cite{azmooden}. In contrast, here some of the coefficients in the Gamma expression are linear polynomials in the target, and the resulting  combination of  the Gamma operators  coincide with a polynomial in the target of degree at least two.
		
		\item[(ii)] Unlike the situation with targets in the first two Wiener chaoses, it is not straightforward (at least to us) how one can translate the LHS of the relations \eqref{eq:Gamma3}, \eqref{eq:Gamma3-1ess-cumulants} and \eqref{eq:Gamma4} in terms of cumulants/moments of the target.   The main obstacle is that the Gamma operators are not stable over Wiener chaoses of order higher than two. 
		
\item[(iii)] In principle the approach used to prove Proposition \ref{prop:Gamma-Type-Expressions} can be used together with the Stein operators of \cite{agg19} for $H_p(X)$, $p\geq5$,
to obtain analogous ``Gamma characterisations" for $H_p(X)$, $p\geq5$. 
Due to the complexity of these Stein operators, this would be quite an involved undertaking, though.
\end{itemize}
\end{rem}

\rg{
\begin{rem}
In this remark, we illustrate a possibility of how our findings can be applied in some concrete probabilistic approximations. In fact, one of our major motivations towards this study comes from the rich and classical asymptotic theory of $U$-statistics \cite{Kor_Ustatistics,Lee_Ustatistics}.    It is well-known that typical degenerate $U$-statistics are asymptotically non-normal. Although, the asymptotic theory of degenerate $U$-statistics are well-understood, very little is known about the rate of convergences. More precisely,   let $(X_n)$ be a sequence of i.i.d.\ random variables with $\E[X_1]=0$, $\E[X^2_1]=1$. Consider the following $U$-statistics having kernel $\psi (x_1,\dots,x_p) =\prod_{i=1}^{p} x_i$, $p\geq1$,
\begin{equation*}
U_n : =   {n \choose p} ^{-1}  \sum_{1 \le i_1 < \dots < i_p \le n} \psi (X_{i_1},\dots,X_{i_p}),  \quad  n\ge 1,
\end{equation*}	
and hence the order of degeneracy is $p-1$. Then, it holds that (see \cite[Chapter 3]{Lee_Ustatistics}):
\begin{equation}\label{eq:LimitUstatistic}
F_n:   =  n^{p/2} U_n   \stackrel{\text{law}}{\longrightarrow}  H_p(X),  \quad  \text{ where } \, X \sim N(0,1).
\end{equation}	
To the best of our knowledge, apart from the particular case $p=2$ (see, e.g., \cite{aaps19a}), there are no results in the existing literature on the quantification of probabilistic limit theorems \eqref{eq:LimitUstatistic} due to the obvious lack of characterising Stein operators for target distributions of the form $H_p(X)$. In order to highlight how the materials in Section \ref{sec4} can be utilised, hereafter, assume in addition that the sequence $(X_n)$ are i.i.d.\ $N(0,1)$ random variables (to locate ourself in a Gaussian setting). Hence, by embedding into an isonormal Gaussian process $\mathbb{X}$ on a suitable Hilbert space $\mathfrak{H}$ having an orthonormal basis $\{h_n\}$, we can write $X_n= I_1 (h_n)$, where $I_1$ is the first Wiener-It\^{o} integral with respect to $\mathbb{X}$.  Let $p=4$. Then, as $n\rightarrow\infty$,
\begin{align*} 
F_n &= \frac{24 n^2}{n(n-1)(n-2)(n-3)}  \sum_{i \le i_1 < i_2<i_3<i_4 \le n} X_{i_1} X_{i_2} X_{i_3} X_{i_4} \\
& =  \frac{24 n^2}{n(n-1)(n-2)(n-3)}  \sum_{i \le i_1 < i_2<i_3<i_4 \le n}   \Big(   I_4 \left(  h_{i_1} \tilde{\otimes} h_{i_2}  \tilde{\otimes} h_{i_3}  \tilde{\otimes} h_{i_4} \right)      \\
&\quad+4 I_2  \left(      \left(    h_{i_1}\tilde{\otimes} h_{i_2}   \right) \tilde{\otimes}_1    \left(     h_{i_3}\tilde{\otimes} h_{i_4}    \right)   \right)   \Big) \\
&  \, \stackrel{\text{law}}{\longrightarrow}  H_4(X).
\end{align*}
One has to note that $F_n \in C_2 \oplus C_4$, $n\in \N$ (the direct sum of the second and fourth Wiener chaoses associated to $\mathbb{X}$).  Denote by $L$ the characterising Stein operator for the target distribution $H_4(X)$ appearing in equation (\ref{h4x}). Let $\mathcal{H}$ be a suitable separating class of test functions; see \cite[Appendix C]{n-p-book}. Consider the associated Stein equation  $Lf (x)    = h (x)  - \E[h (H_4(X))]$ for a given test function $h \in \mathcal{H}$ (note that this is a third order non-homogeneous ODE).  For a moment, assume that, the  ODE admits a unique three times differentiable solution $f=f_h$ such that 
\begin{equation}\label{eq:SteinUniversality}
\sup_{h\in \mathcal{H}}  \Vert  f^{(k)} \Vert_\infty <C, \quad  k=0,1,2,3,
\end{equation}
for some constant $C>0$. Bounds \eqref{eq:SteinUniversality} above are known as Stein universality phenomenon in the Stein's method literatures and need to be verified for each characterising Stein operator corresponding to a given target distribution (for example,  when the target distribution is standard Gaussian $X \sim N(0,1)$, Stein universal bounds are well-understood, see \cite{Daly}). Then, by using the Malliavin integration by parts formula \eqref{mibp}, one can obtain 
\begin{align*}
d_{\mathcal{H}}   & \left(   F_n , H_4 (X)\right)    =\sup_{h \in \mathcal{H}}  \big \vert   \E[h (F_n)]    - \E[h(H_4(X))] 
  \big \vert  =  \sup_{h \in \mathcal{H}}  \big \vert       Lf (F_n )   \big \vert   \\
  & \le C  \Bigg\{              \big \vert  \kappa_2 (F_n)  - \kappa_2 (H_4(X))    \big \vert          + \big \vert  \kappa_3 (F_n)  - \kappa_3(H_4(X))    \big \vert        \\
  &  \qquad    +   \sqrt{    \E  \Big[    	\Gamma_3 (F_n) - 60 \Gamma_2(F_n) + 16 (9-F_n) \Gamma_1(F_n) - 192 (F_n+6)(3-F_n)     \Big]^2 }         \Bigg\}.
\end{align*}	
Therefore, a bound on the rate of convergence can be computed (although with cumbersome calculations) by analysing the last summand involving Malliavin $\Gamma$ operators up to order three. It remains to obtain suitable bounds for the quantities $\|f^{(k)}\|$, $k=0,1,2,3$, which we leave as an interesting and non-trivial open problem.

 \end{rem}	
}

\appendix

\section{Stein operators for univariate Gaussian Hermite polynomials}\label{appendixa}

Let $X\sim N(0,1)$. All Stein operators in this appendix were obtained by \cite{agg19}. The Stein operators (\ref{h3x}) and (\ref{h4x}) were also earlier obtained by \cite{gaunt-34}.
We do not reproduce the Stein operators of \cite{agg19} for $H_7(X)$ and $H_8(X)$; we refer the reader to Appendix B of arXiv version no.\ 1 of \cite{agg19} for their complicated formulas.


\vspace{3mm}


\noindent{$H_3(X)$ :}
\begin{align} \label{h3x-new}
5\, y -(3\,y^2+12)\partial   +207\,y\partial^{2}
   +(351\,y^2-1080)\partial^{3}
   +(81\,y^3-324\,y)\partial^{4}
\end{align}
\begin{align}\label{h3x}
   y -6\partial
   -99\,y\partial^{2}
   +(216-27\,y^2)\partial^{3}
   +486\,y\partial^{4}
   +(486\,y^2-1944)\partial^{5} 
\end{align}

\noindent{$H_4(X)$ : }
 \begin{align}\label{h4gen}
  ( -y^2+50\,y+24 )
 + (64\,y^2+72\,y-1008 )\partial
 + (16\,y^3-48\,y^2-576\,y+1728 )\partial^{2}
  \end{align} 
\begin{align} \label{h4x}
y -(24+44\,y) \partial  +(576+144\,y-16\,y^2)\partial^{2} +(192\,y^2+576\,y-3456)\partial^{3}
\end{align}

\noindent{$H_5(X)$ :} 
\begin{align}&y- 120\partial- 75325\,y\partial^2+ (- 81875\,y^2+7704000 )\partial^3+ (- 31250\,y^3+270600000\,y)\partial^4\nonumber\\
&+ (- 3125\,y^4 + 527800000\,y^2 - 39086400000)\partial^5+ (280000000\,y^3 - 155065000000\,y)\partial^6 \nonumber\\
&+ (35000000\,y^4 - 241335000000\,y^2 + 14306880000000)\partial^7\nonumber\\
&+ (- 198750000000\,y^3+53403600000000\,y )\partial^8\nonumber\\
&+ (- 33125000000\,y^4 + 34950000000000\,y^2 - 1170432000000000)\partial^9\nonumber\\
&+ (39000000000000\,y^3 - 10843200000000000\,y)\partial^{10}\nonumber\\
& + (9750000000000\,y^4 - 6696000000000000\,y^2 + 352512000000000000)\partial^{11}\nonumber\\
& + (- 2160000000000000\,y^3+622080000000000000\,y )\partial^{12}\nonumber\\
\label{h5x}&+(- 1080000000000000\,y^4 + 622080000000000000\,y^2 - 29859840000000000000)\partial^{13}
\end{align}

\noindent{$H_6(X)$ :}
\begin{align} &
y + (- 1278\,y - 720)\partial + (- 972\,y^2 + 103320\,y + 756000)\partial^2 \nonumber \\ &
+ (- 216\,y^3 + 228960\,y^2 + 16491600\,y - 120528000)\partial^3 \nonumber \\ &
+   (71280\,y^3 + 6771600\,y^2 - 307152000\,y - 3265920000)\partial^4  \nonumber \\ &
+   (- 314928000\,y^2 - 19945440000\,y + 125971200000)\partial^5 \nonumber \\ 
\label{h6x}&+   (- 209952000\,y^3 - 19945440000\,y^2 + 251942400000\,y + 7558272000000)\partial^6     
   \end{align} 



\subsection*{Acknowledgements}
RG was supported by a Dame Kathleen Ollerenshaw Research Fellowship. \rg{We would like to thank the reviewers for their helpful comments and suggestions.}

\end{document}